\documentclass[11pt,a4paper]{amsart}
\usepackage{amssymb,amsmath}
\usepackage{graphicx, tikz,caption}
\usepackage{hyperref}

\newtheorem{theorem}{Theorem}[section]
\newtheorem{prop}[theorem]{Proposition}
\newtheorem{problem}[theorem]{Problem}

\newtheorem{lemma}[theorem]{Lemma}
\newtheorem{remark}[theorem]{Remark}
\newtheorem{definition}[theorem]{Definition}
\newtheorem{example}[theorem]{Example}
\newenvironment{proofof}[1]{\noindent {\em Proof of #1.}}{ \hfill $\Box$\\ }

\begin{document}

\title{Rotated Odometers}
\author{Henk Bruin}
\address{Henk Bruin, Faculty of Mathematics, University of Vienna, Oskar-Morgenstern-Platz 1, 1090 Vienna, Austria}
\email {henk.bruin@univie.ac.at} 
\author{Olga Lukina}
\address{Olga Lukina, Faculty of Mathematics, University of Vienna, Oskar-Morgenstern-Platz 1, 1090 Vienna, Austria, and Mathematical Institute, Leiden University, P.O. Box 9512, 2300 RA Leiden, The Netherlands}
\email{o.lukina@math.leidenuniv.nl}
\thanks{HB and OL were supported by the FWF Project P31950-N35}
\date{December 16, 2022}

\maketitle

\begin{abstract}
We describe the infinite interval exchange transformations, called the rotated odometers, that are obtained as compositions of finite 
interval exchange transformations and the von Neumann-Kakutani map. 
We show that with respect to Lebesgue measure on the unit interval, every such transformation 
is measurably isomorphic to the first return map of a rational parallel 
flow on a translation surface of finite area with infinite genus and a finite number of ends. 
We describe the dynamics of rotated odometers by means of 
Bratteli-Vershik systems, derive
several of their topological and ergodic properties, and investigate in detail a range of specific examples of rotated odometers.
\end{abstract}

\section{Introduction}\label{sec:intro}

In this paper, we consider a family of infinite interval exchange transformations (IETs) which arise as perturbations of the 
von Neumann-Kakutani map of the unit interval, and as the first return maps of flows of rational slope on 
certain flat surfaces of infinite genus. 
We study the dynamical and ergodic properties of the maps in this family.

Each map in this family has a unique aperiodic minimal subsystem, and thus this family presents a class of naturally arising systems 
with this property.
We show that each aperiodic (not necessarily minimal) subsystem 
is measurably isomorphic to a Bratteli-Vershik system on a Cantor set, 
and study its ergodic invariant measures and the spectrum of its Koopman operator. 
We construct  infinitely many examples, where each minimal set has the dyadic odometer as a factor, and infinitely many examples 
where each minimal set does not have the dyadic odometer as a factor, but it is not weakly mixing.

\subsection*{Rotated odometers}
The von Neumann-Kakutani map $\mathfrak a: I \to I$, where $I = [0,1)$ is the half-open unit interval, is given by the formula
\begin{align}\label{eq-odometer}
\mathfrak a(x) =  x - (1-3 \cdot 2^{1-n}) \qquad  \text{ if } x \in I_n =  [1-2^{1-n}, 1-2^{-n}),\ n \geq 1.
\end{align}
In words, it re-arranges the interval partition $\{ I_n \}_{n \geq 1}$ of $I$ in the opposite order, see Figure~\ref{fig:lochness}(a). We
divide the interval $I = [0,1)$ into $q$ half-open subintervals of length $\frac{1}{q}$, and we  let $\pi$ be any permutation of $q$ symbols. Let $R_\pi: I \to I$ to be the map which permutes these $q$ subintervals according to $\pi$. Then
\begin{align}\label{eq-rotod}
   F_\pi = \mathfrak a \circ R_\pi: I \to I 
\end{align}
  is an infinite IET called a \emph{rotated odometer}. The term `rotated odometer' was introduced since for some permutations $\pi$, the map $R_\pi:I \to I$ may be seen as a rotation on the unit circle.

\subsection*{Infinite genus surfaces}
Let $S$ be the unit square, and identify its vertical edges by a single translation (as if creating a cylinder), and its horizontal edges by the von Neumann-Kakutani map $\mathfrak a$, to obtain the surface $L$, see Section ~\ref{sec:flow} for details. In order to make the identifications work we must remove a countable number of points from the horizontal edges, and as a result the surface $L$ is non-compact. The removed points are identified into a single point, therefore the resulting surface has a single \emph{end}, i.e., a distinct way to go to infinity, see Figure~\ref{fig:LNM1}. The surface $L$ has unit area, and infinite genus. Topological surfaces of this type are called \emph{Loch Ness monsters}, and they have appeared in the literature as leaves in foliations by surfaces \cite{PS1981,Ghys1995}. Loch Ness monsters also appear as translation surfaces with infinite angle or wild singularities, such as the \emph{Chamanara} or the \emph{baker's surface} \cite{Cham2004,DHP,Rbook},
or the \emph{infinite staircase} \cite{DHP}. An interesting family of infinite-type translation surfaces was constructed in \cite{LT2016}. The constructions of the families of surfaces in \cite{LT2016} and in our paper are reminiscent of that of the Chamanara surface in \cite{DHP}. However, in our paper the Loch Ness monsters lack certain metric symmetries which are present in \cite{DHP}, 
and so the methods used to study the properties of the latter in \cite{DHP}, are not applicable in our case.

Consider the flow lines on the square $S$ which are at the constant angle $\theta = \tan^{-1}(q/p)$ with the horizontal, 
where $p,q \in \mathbb{Z} \setminus \{ 0\}$. When the flow lines traverse the square from the bottom to the top, they travel through the horizontal distance $p/q$. Therefore, the first return map to the horizontal section in the surface $L$, corresponding to the horizontal edges of $S$, is the composition of a translation by $p/q$ and the von Neumann-Kakutani map $\mathfrak a$, i.e., a rotated odometer. Then a natural question is, can any rotated odometer 
\eqref{eq-rotod}, i.e., for an arbitrary permutation $\pi$, be realized as the first return map of a flow on a Loch Ness monster? Our first Theorem~\ref{thm-main0} below states that the answer is yes, provided we can make mild modifications to the topology of the surface.

We denote one-dimensional Lebesgue measure by $\lambda$.

\begin{theorem}\label{thm-main0}
Let $q \geq 2$ and let $\pi$ be a permutation of $q$ symbols, and let $p \geq q$ be an integer. Then there exists a translation 
surface $L_{\pi,p}$ obtained by identifying by translations the sides of the unit square with countable number of boundary points removed, which has the following properties:
\begin{enumerate}
\item The surface $L_{\pi,p}$ has finite area, one non-planar end and at most a finite number of planar ends.
\item The metric completion of $L_{\pi,p}$ contains a single wild singularity and at most a finite number of cone angle singularities.
\item There exists a section $P \subset L_{\pi,p}$ parallel to the horizontal edge of the unit square with 
Poincar\'e map $F:P \to P$ of the flow of rational slope $q/p$, such that $(P,F,\lambda)$ is measurably isomorphic to the rotated odometer $(I,F_\pi,\lambda)$.
\end{enumerate}
\end{theorem}

The technical notions in the statement of Theorem~\ref{thm-main0} are explained rigorously in Section~\ref{sec:flow}, where this theorem is proved. Intuitively, singularities in this theorem correspond to the points we must remove from $S$ when identifying edges in order to obtain a surface where every point has a Euclidean neighborhood. Distinct removed points may be identified into a single singularity. Each singularity results in a puncture in the surface $L$, and each puncture corresponds to an end of $L$. The notions of a planar or a non-planar end describe the topology of a neighborhood of an end, namely, if an end is non-planar, then every such neighborhood has infinite genus.

A finite area surface with infinite genus, one non-planar end and two planar ends is depicted in Figure~\ref{fig:LNM2}. 
We call a Loch Ness monster with additional planar ends a \emph{Loch Ness monster with whiskers}. 
A surface of this type is described in Example~\ref{eq-whiskers}.

\begin{figure}
\centering
\begin{minipage}{.45\linewidth}
          \includegraphics[width=5cm]{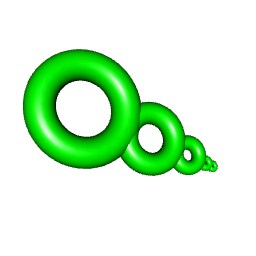}
            \captionof{figure}{The finite area Loch Ness monster}
             \label{fig:LNM1}  
        \end{minipage}    
      \begin{minipage}{.45\linewidth}  
          \includegraphics[width=5cm]{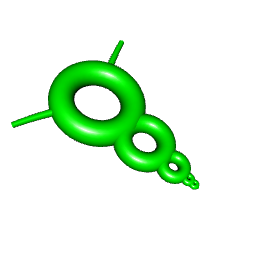}
   \captionof{figure}{The Loch Ness monster with two whiskers}
    \label{fig:LNM2}  
        \end{minipage}      
\end{figure}

\subsection*{Dynamical properties of rotated odometers}
We now study in detail the dynamics of the rotated odometer $(I,F_\pi,\lambda)$ for any $q \geq 2$ 
and any permutation $\pi$ on $q$ symbols. Such a map can be considered as a perturbation of the von Neumann-Kakutani map $\mathfrak a$. 
From this point of view, it is natural to ask, which properties of the von Neumann-Kakutani map are preserved under such perturbation. We show that even in this highly controlled setting, much of the inner structure of the von Neumann-Kakutani system can be destroyed by a perturbation, although some features are preserved.

The first basic result is that the minimality of $\mathfrak a$ may be destroyed, but the minimal subset of the aperiodic subsystem is always unique. Recall that $I=[0,1)$. 

\begin{theorem}\label{thm-main1}
There exists a decomposition $I = I_{per} \cup I_{np}$ with the following properties:
\begin{enumerate}
\item Every point in $I_{per}$ is periodic, the restriction $F_\pi: I_{per} \to I_{per}$ is well-defined and invertible.
\item If $I_{per}$ is non-empty, then $I_{per}$ is a (possibly infinite) union of half-open intervals $[x,y)$.
\item The set $I_{np}$ contains $0$, and $F_\pi: I_{np} \to I_{np}$ is well-defined and  invertible at every point except $0$.
\item There is a unique minimal subsystem $(I_{min}, F_\pi) $ of $(I_{np}, F_\pi)$, and $0 \in I_{min}$.
\end{enumerate}
\end{theorem}

Theorem~\ref{thm-main1} is proved in Section~\ref{sec:general}.
Examples~\ref{ex:120},~\ref{ex-countableperiodic} and~\ref{ex-finiteperiodic} show that $I_{per}$ may be an empty set, or a finite or infinite union of half-open intervals. An infinite IET which contains an infinite collection of intervals of periodic points was also considered in \cite{HRR}, see Example~\ref{ex-distortedodometer}. 

\begin{remark}\label{remark-uniq}
{\rm
The existence of a unique minimal aperiodic set imposes strong restrictions on the behavior of the system. For instance, as shown in \cite{HPS1992}, certain $C^*$-algebras associated to systems with unique minimal sets on zero-dimensional spaces can be exhibited as cross products of an abelian $C^*$-algebra by a single homeomorphism. We refer the reader to \cite{HPS1992} for more on $C^*$-algebras and dimension groups in this setting. 
The rotated odometers considered in this paper, provide a naturally arising family of examples of dynamical systems with unique minimal sets; systems with this property are not readily found in the literature. 
This is another motivation to study rotated odometers.
}
\end{remark}

\subsection*{Ergodic properties of rotated odometers}
The rotated odometer map $F_\pi$ acts on $I$ by piecewise translations and so preserves 
Lebesgue measure $\lambda$ on $I$. Moreover, in Section~\ref{sec:lebesgueergod} we prove:

\begin{theorem}\label{thm-nonergodicleb}
Lebesgue measure is ergodic for $(I,F_\pi)$ if and only if there are no periodic points.
\end{theorem}

One implication in Theorem~\ref{thm-nonergodicleb} is immediate, the other one requires work. Another natural property of all rotated odometers is that they have zero topological entropy. 
The proof of Theorem~\ref{thm-entropy-intr} below can be found in Section~\ref{subsec-entropy}.

\begin{theorem}\label{thm-entropy-intr}
For any $q \geq 1$ and any permutation $\pi$ of $q$ symbols, the topological entropy $h_{top}(F_\pi) = 0$.
\end{theorem}

To further study the dynamics of the aperiodic subsystem $(I_{np}, F_\pi)$ we use the standard technique of doubling points in the orbits of discontinuities to embed $(I_{np}, F_\pi)$ into a dynamical system $(I_{np}^*,F_\pi^*)$ given by a homeomorphism $F_\pi^*$ of a Cantor set $I_{np}^*$. The procedure is described in detail in Section \ref{sec-compactification}, 
with main results summarized in Theorem~\ref{thm-compactify}. Since the Cantor set $I_{np}^*$ is obtained by adding to $I_{np}$ a countable collection of points, there is a correspondence of invariant measures on the Cantor system and on $(I_{np},F_\pi)$. 
We show in Section~\ref{sec:BV} that the Cantor system $(I_{np}^*,F_\pi^*)$ is conjugate to the Bratteli-Vershik system on an eventually stationary Bratteli diagram, see Theorem~\ref{thm-main3l}.

Bratteli-Vershik systems are a powerful tool to study the dynamics of maps of Cantor sets, described in many sources, see for instance \cite{BKMS2010,BKMS2013,HPS1992} and references therein. 
In the rest of the paper, we use Bratteli-Vershik systems to study the number of ergodic invariant measures on $(I_{np},F_\pi)$, and the discrete spectrum of the Koopman operator for different ergodic measures.

Since the Bratteli-Vershik diagram conjugate to $(I_{np}^*,F_\pi^*)$ is associated to a pre-periodic sequence of substitutions on $q$ letters (see Section~\ref{sec:BV}), we have the following result.

\begin{theorem}\label{thm-ergmeasures}
For any $q \geq 1$ and any permutation $\pi$ of $q$ symbols, the aperiodic subsystem $(I_{np},F_\pi)$ admits at most $q$ ergodic invariant measures, 
and its unique minimal subsystem $(I^*_{min},F_\pi^*)$ is uniquely ergodic.
\end{theorem}

\subsection*{Factors of rotated odometers} The next series of results is motivated by considering rotated odometers as permutations of the von Neumann-Kakutani map, which is known to be measurably isomorphic to the dyadic odometer. We know from Theorem~\ref{thm-main1} that the dynamical characteristics of the dyadic odometer, such as, for instance, minimality, may be destroyed by a perturbation to a rotated odometer. 
Therefore, it is natural to ask, whether they are preserved at least at the level of factors, i.e., whether the dyadic odometer is still a measure-theoretical or a topological factor of the rotated odometer. We ask this question for both the aperiodic system $(I_{np}^*,F_\pi^*)$ and for its unique minimal set $(I_{min}^*,F_\pi^*)$.

\begin{theorem}\label{thm-main4}
Let $(I_{min}^*, F_\pi^*)$ be the minimal subset of $(I_{np}^*,F_\pi^*)$. Then:
\begin{enumerate}
\item \label{main4-1} There exist infinitely many $q \geq 3$, and permutations $\pi$ of $q$ symbols, such that the minimal 
system $(I_{min}^*, F_\pi^*)$ has a dyadic odometer as a factor.
\item  \label{main4-2} There exist infinitely many $q \geq 3$, and permutations $\pi$ of $q$ symbols, 
such that the minimal system $(I_{min}^*, F_\pi^*)$ does not factor to a dyadic odometer, and is not weakly mixing.
\end{enumerate}
\end{theorem}

Theorem~\ref{thm-main4} is proved in Section~\ref{subsec-proof45}. 

Host~\cite{Host86} proved that for substitution shifts, the measure-theoretical and topological
rotational factors coincide, and so in Theorem~\ref{thm-main4} by a factor we mean either of them. The measure implicitly used in this theorem is the unique ergodic measure supported on the minimal set.

In both statements of Theorem~\ref{thm-main5} below, Lebesgue measure on $I$ is ergodic for the rotated odometer $(I,F_\pi)$, and $I_{np} = I$.
The measure $\lambda$ is the pushforward of the Lebesgue measure on $I_{np}$ to $I_{np}^*$ by the embedding map.

\begin{theorem}\label{thm-main5}
Let $(I_{np}^*,F_\pi^*)$ be the aperiodic subsystem of a rotated odometer. Then:
\begin{enumerate}
\item \label{main5-1} If $q = 5$ and $\pi=(01234)$, then the rotated odometer $(I_{np}^*,F_\pi^*,\lambda)$ has the dyadic odometer as the maximal equicontinuous factor, and the factor map is continuous.
\item \label{main5-2} If $q = 5$ and $\pi=(02431)$, then the rotated odometer $(I_{np}^*,F_\pi^*,\lambda)$ has the cyclic group of order $4$ as the maximal equicontinuous factor, but the factor map is not continuous.
\end{enumerate}
\end{theorem}

Theorem~\ref{thm-main5} is proved in Section~\ref{subsec-proof45}.

\subsection*{Open problems} 
We show in Theorem~\ref{thm-main0} that rotated odometers can be considered as first return maps of flows of rational slope, on certain translation surfaces of finite area and infinite genus with a finite number of ends.
The following question is natural.

\begin{problem}\label{prob-1}
Find a Bratteli-Vershik system that models the Poincar\'e map of a flow of irrational slope on a translation surface of finite area with infinite genus and finite number of ends.
\end{problem}

The next open problem stems from Theorem~\ref{thm-main4} whose proof is constructive. At the moment we are not aware of a general condition which would ensure that a rotated odometer or its minimal set has, or does not have, the dyadic odometer as a factor. We pose this as an open question.

\begin{problem}\label{prob-2}
Let $F_\pi = \mathfrak a \circ R_\pi: I \to I$ be a rotated odometer. 
Find necessary and sufficient conditions under which $(I^*_{min},F^*_\pi)$ has a dyadic odometer as a factor.
\end{problem}

In the system described in Statement \ref{main5-2} of Theorem~\ref{thm-main5}, neither the minimal subsystem with respect to the unique ergodic measure, nor the aperiodic system $(I_{np}^*,F^*_\pi)$ with respect to Lebesgue measure, have the dyadic odometer as factor. 
Instead, the minimal subsystem has an irrational eigenvalue, while the full rotated odometer factors onto the cyclic group with four elements. Therefore, the following question is natural.

\begin{problem}\label{prob-4}
Are there any examples in our class of rotated odometers for which the minimal subsystem
$(I^*_{min}, F_{\pi}^*)$, or the aperiodic system $(I^*_{np},F^*_\pi)$ is weakly mixing?
\end{problem}

{\bf Acknowledgements:} The authors thank Ian Putnam for drawing their attention to \cite{HPS1992} 
and the significance of the property of a dynamical system having a unique minimal set.
We thank the referee for numerous suggestions to improve the exposition and readability of this paper.

\section{Rotated odometers and flows on translation surfaces}\label{sec:flow}

In this section we recall the basic properties of infinite translation surfaces and prove Theorem~\ref{thm-main0}.

\subsection{Loch Ness monsters}\label{subsec:LNM}
Consider the unit square $S$ without corner points. The upper and the lower sides of $S$ are identified with the interior $(0,1)$ of $I = [0,1)$. The vertical sides of $S$ are identified with $J = (0,1)$.
We make the identification
$$(x,1) \sim_h (\mathfrak a(x),0),$$ 
see Figure~\ref{fig:lochness}(a), where  
$I_k = \left[1 - \frac{1}{2^{k-1}},1-\frac{1}{2^k}\right)$,  $k \geq 1$, are the intervals of continuity of $\mathfrak a$. We identify the vertical 
sides using the equivalence relation $(1,y) \sim_v (0,y)$ as in the standard torus. 

\begin{figure}[ht]
\begin{center}
\begin{tikzpicture}[scale=0.7]
\draw[-, draw=black] (0,8)--(7.6,8)node[pos=0.25,anchor=north]{$I_1$} 
node[pos=0.66,anchor=north]{$I_2$} node[pos=0.85,anchor=north]{$I_3$} 
node[pos=0.95,anchor=north]{$I_4$} node[pos=0,anchor=east]{$1$}; 
\draw[dotted, draw=black] (7.6,8)--(8,8); 
\draw[-, draw=black] (0,8)--(0,0)node[pos=0.5,anchor=east]{$J_1$}; 
\draw[-, draw=black] (8,0)--(8,8)node[pos=0.5,anchor=west]{$J_1$}; 
\draw[-, draw=black] (0.4,0)--(8,0) node[pos=0.75,anchor=south]{$\mathfrak a(I_1)$} 
node[pos=0.36,anchor=south]{$\mathfrak a(I_2)$} node[pos=0.17,anchor=south]{$\mathfrak a(I_3)$} 
node[pos=0.02,anchor=south]{$\mathfrak a(I_4)$}node[pos=-0.04,anchor=east]{$0$} node[pos=1,anchor=west]{$1$}; 
\draw[dotted, draw=black] (0,0)--(0.4,0); 
\draw[-,draw=black] (4,-0.1)--(4,0.1);
\draw[-,draw=black] (2,-0.1)--(2,0.1);
\draw[-,draw=black] (1,-0.1)--(1,0.1);
\draw[-,draw=black] (0.5,-0.1)--(0.5,0.1);
\draw[-,draw=black] (4,7.9)--(4,8.1);
\draw[-,draw=black] (6,7.9)--(6,8.1);
\draw[-,draw=black] (7,7.9)--(7,8.1);
\draw[-,draw=black] (7.5,7.9)--(7.5,8.1);
\draw[-,draw=black] (4,0) circle (0.15);
\draw[-,draw=black] (1,0) circle (0.15);
\draw[-,draw=black] (6,8) circle (0.15);
\draw[-,draw=black] (0,8) circle (0.15);
\draw[-,draw=black] (7.5,8) circle (0.15);
\draw[-,draw=black] (3.85,7.85) rectangle (4.15,8.15);
\draw[-,draw=black] (6.85,7.85) rectangle (7.15,8.15);
\draw[-,draw=black] (1.85,-0.15) rectangle (2.15,0.15);
\draw[-,draw=black] (0.35,-0.15) rectangle (0.65,0.15);
\draw[-,draw=black] (7.85,-0.15) rectangle (8.15,0.15);
%%%%
%%%%
\draw[-, draw=black] (10,8)--(17.6,8) node[pos=0,anchor=east]{$1$}; 
\draw[dotted, draw=black] (17.6,8)--(18,8); 
\draw[-, draw=black] (10,8)--(10,0); 
\draw[-, draw=black] (18,0)--(18,8); 
\draw[-, draw=black] (10.4,0)--(18,0) node[pos=-0.04,anchor=east]{$0$}node[pos=1,anchor=west]{$1$}; 
\draw[dotted, draw=black] (10,0)--(10.4,0); 
\draw[-,draw=black] (14,-0.1)--(14,0.1);
\draw[-,draw=black] (12,-0.1)--(12,0.1);
\draw[-,draw=black] (11,-0.1)--(11,0.1);
\draw[-,draw=black] (10.5,-0.1)--(10.5,0.1);
\draw[-,draw=black] (14,7.9)--(14,8.1);
\draw[-,draw=black] (16,7.9)--(16,8.1);
\draw[-,draw=black] (17,7.9)--(17,8.1);
\draw[-,draw=black] (17.5,7.9)--(17.5,8.1);
\draw[bend right=45, dashed, draw=black] (12,8) to (15,8);
\draw[bend right=45, dashed, draw=black] (16,0) to (13,0);
\draw[bend right=60, dashed, draw=black] (16.5,8) to (17.25,8);
\draw[bend right=60, dashed, draw=black] (11.5,0) to (10.75,0);
\end{tikzpicture}
\caption{(a) Identifications of the horizontal sides of the unit square by the von Neumann-Kakutani map, and of the vertical sides by translations. Circles and squares represent identifications of limit points 
in $\overline{L}$. (b) Non-separating curves in $L$ are represented by dashed lines.}
\label{fig:lochness}
\end{center}
\end{figure}
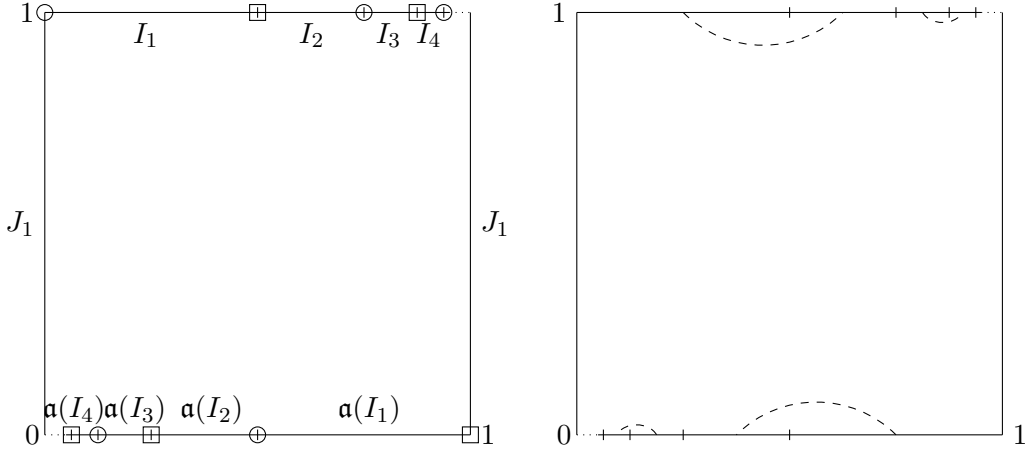 

Consider the set of discontinuity points of $\mathfrak a$ in the upper horizontal edge, 
and of their images under $\mathfrak a$ on the lower horizontal edge, given by
  $$D = \{(1 - 2^{-k},1), (\mathfrak a(1 - 2^{-k}),0) : k \geq 1\}.$$ 
Define the non-compact surface $L$ by applying the equivalence relations,
\begin{align}\label{eq-quot}L = (S \setminus D)/ \sim_h,\sim_v. \end{align}
The construction of $L$ above is similar to that of the \emph{Chamanara surface} in the literature \cite{Cham2004,DHP,Rbook}, except 
that the identification of the vertical sides of the square $S$ in the Chamanara surface is done using the 
von Neumann-Kakutani map. 

Non-compact surfaces are classified up to a homeomorphism by their genus and the space of ends. 
Intuitively, an \emph{end} of a surface is a distinct way to go to infinity in the surface. Adding ends 
to a surface can be considered as its compactification \cite{DHP,CCbook}. A surface with one end and infinite 
genus can be pictured as the Euclidean plane with an infinite number of handles attached, and this is the reason it was called the \emph{Loch Ness monster} in \cite{PS1981}. The Euclidean plane has infinite area. The surface in Figure~\ref{fig:LNM1} 
is homeomorphic to the plane with an infinite number of handles attached, and it has unit area.
 
An open neighborhood of an end is a surface in its own right, and so one can talk about the genus of this surface. An end $e$ is \emph{planar} if its has a neighborhood of genus zero, and $e$ is called \emph{non-planar} otherwise. 
The single end of the Loch Ness monster in Figure~\ref{fig:LNM1} is non-planar.

Arguments similar to the one for the Chamanara surface in \cite{Rbook} show that $L$ is a translation surface of finite area and 
infinite genus with one non-planar end. We sketch the proof in Proposition \ref{prop-LNmonster} for completeness.

Let $\sigma \in \overline L$ be a singularity, and let $B_\epsilon$ be an open neighborhood of $\sigma$ in $\overline{L}$ of radius $\epsilon > 0$. A singularity $\sigma$ is \emph{wild}, if there is no finite or infinite cyclic translation covering from $B_\epsilon \setminus \sigma$, to a once-punctured disc $B(0,\epsilon) \setminus \{0\} \subset \mathbb{R}^2$ for any $\epsilon > 0$. Recall that a \emph{saddle connection} is a geodesic in $\overline{L}$ which joins two not necessarily 
distinct singularities in $\overline{L}$, and which does not contain a singularity in its interior. In particular, $\sigma$ is wild if for any $\epsilon>0$ the neighborhood $B_\epsilon$ contains an infinite number of saddle connections.

A closed curve $\gamma: \mathbb{S}^1 \to L$ is \emph{non-separating} if 
$L \setminus \gamma(\mathbb{S}^1)$ is connected, where $\mathbb{S}^1$ is the circle of unit length. A surface $L$ has genus $g$, if the maximum cardinality of a set of disjoint non-separating curves in $L$ is $g$. If $L$ admits an infinite number of 
disjoint non-separating curves, 
then it has infinite genus.

\begin{prop}\label{prop-LNmonster}
The surface $L$ is a translation surface of finite area and infinite genus, and the metric completion $\overline{L}$ of $L$ contains a single wild singularity. 
Thus $L$ is a Loch Ness monster, that is, $L$ is an infinite genus surface with one non-planar end.
\end{prop}

\begin{proof} Recall from \cite{Rbook} that a translation surface is a surface which admits an atlas where the transition
functions are locally translations. The surface $L$ in \eqref{eq-quot} satisfies this definition since $\mathfrak a$ is an IET and so 
identifications are by translations. The metric completion $\overline{L}$ of $L$ 
is clearly compact, and we claim that $\overline{L} \setminus L$ is a single point. To see that we use a similar argument 
to the one in \cite[Example 1.15]{Rbook} for the Chamanara surface. Namely, the identifications of the horizontal 
edges by $\mathfrak a$ induce the identification between the limit points in the metric completion marked by circles and squares in 
Figure~\ref{fig:lochness}(a), and we conclude that there is no more than two points in $\overline L \setminus L$, 
one marked by circles and another one by squares. Since the distance between these two points is not bounded 
away from zero, they are the same point, and $\overline{L} \setminus L$ consists of a single singularity denoted by $\sigma$. We notice that for every $\epsilon>0$, the part of the neighborhood $B_\epsilon$ of $\sigma$ near the corner points of $S$ contains an infinite number of saddle connections (horizontal segments on the upper and the lower 
edges of the square). Thus $\sigma$ is a wild singularity.

\iffalse
\begin{figure}[ht]
\begin{center}
\begin{tikzpicture}[scale=0.6]
\draw[-, draw=black] (0,8)--(7.6,8) node[pos=0,anchor=east]{$1$};  
\draw[dotted, draw=black] (7.6,8)--(8,8); 
\draw[-, draw=black] (0,8)--(0,0); 
\draw[-, draw=black] (8,0)--(8,8); 
\draw[-, draw=black] (0.4,0)--(8,0) node[pos=-0.04,anchor=east]{$0$}node[pos=1,anchor=west]{$1$}; 
\draw[dotted, draw=black] (0,0)--(0.4,0); 
\draw[-,draw=black] (4,-0.1)--(4,0.1);
\draw[-,draw=black] (2,-0.1)--(2,0.1);
\draw[-,draw=black] (1,-0.1)--(1,0.1);
\draw[-,draw=black] (0.5,-0.1)--(0.5,0.1);
\draw[-,draw=black] (4,7.9)--(4,8.1);
\draw[-,draw=black] (6,7.9)--(6,8.1);
\draw[-,draw=black] (7,7.9)--(7,8.1);
\draw[-,draw=black] (7.5,7.9)--(7.5,8.1);
\draw[bend right=45, dashed, draw=black] (2,8) to (5,8);
\draw[bend right=45, dashed, draw=black] (6,0) to (3,0);
\draw[bend right=60, dashed, draw=black] (6.5,8) to (7.25,8);
\draw[bend right=60, dashed, draw=black] (1.5,0) to (0.75,0);
\end{tikzpicture}
\caption{Non-separating curves in $L$ are represented by dashed lines.}
\label{fig:nonseparating}
\end{center}
\end{figure} 
\fi

To see that $L$ has infinite genus, for $k \geq 1$, let $\gamma_k$ be a closed curve joining the middle point of the interval $I_{2k}$ to the 
middle point of the interval $I_{2k-1}$ lying below the upper horizontal edge of $L$, and then joining the 
middle point of $\mathfrak a(I_{2k-1})$ with the middle point of $\mathfrak a(I_{2k})$ and passing above the lower horizontal 
edge of $L$, see Figure~\ref{fig:lochness}(b). All such curves are disjoint. To see that the complement 
of $\gamma_k$ in $L$ is connected, note that every point in $L \setminus \gamma_k(\mathbb{S}^1)$ is connected to the 
middle point of the square by a continuous path. This shows that $L$ admits an infinite number of non-separating 
curves, and so has infinite genus. 
To show that $L$ has one end
we refer to \cite[Proposition 3.10]{Rbook}, where it is shown that if $L$ is a translation surface such 
that the metric completion $\overline{L}$ is compact, then the space of ends of $L$ is finite and it is 
in one-to-one correspondence with the set of singularities. Since $L$ has a single singularity, it has one end. 
It is clear from the picture and the arguments above that every $\epsilon$-neighborhood of $\sigma$ contains 
an infinite number of non-separating curves, which implies that the single end of the surface $L$ is non-planar.
\end{proof}

\subsection{Flows on the Loch Ness monster}\label{subsec:flowsLNM}
Every point $x \in L$ is contained in a chart of a maximal 
atlas whose transition maps are translations. 
Thus the tangent bundle $TL = \bigcup_{x \in L}T_xL$ carries a flat connection.
For any $x \in L$ the
exponential map ${\rm exp}_x:T_x L \to L$ is well-defined on an open neighborhood of $0$ in $T_xL$ depending on $x$. 
For any angle $\theta \in \mathbb{S}^1$ there is a vector field 
$\mathcal{X}_\theta$ on $L$, whose flow lines are straight lines which make the angle $\theta$ with the horizontal. 
Some flow lines of $\mathcal{X}_\theta$ reach the singular point $\sigma$ in finite time, and so they are not defined for all $t \in \mathbb{R}$. Since the map $\mathfrak a$ has a countable discontinuity set, 
there is at most a countable number of such flow lines. Let $L_\theta$ be the union of flow lines that are defined for all $t \in \mathbb{R}$,
and denote the flow by $\varphi_\theta: \mathbb{R} \times L_\theta \to L_\theta$.

Let $q \geq 2$ and $p \geq 1$ be integers, and let $\theta = \tan^{-1}(q/p)$. Let $P$ be the image of the lower 
horizontal edge of the unit square under \eqref{eq-quot} in $L_\theta$. Then $P$ is a 
Poincar\'e section for the flow lines of the vector field $\mathcal{X}_\theta$, and we denote 
the Poincar\'e map of the flow by $F: P \to P$.

Consider the lift of the flow to the unit square $S$. For any $x$ in the lower edge, consider the flow line 
through $x$. While traveling from the lower edge to the upper edge, the flow line through $x$ moves
in the horizontal direction by distance $p/q$, possibly traversing $S$ a few times. Divide the unit interval $I$ into $q$ 
subintervals of equal length, inducing subdivisions of the lower and upper edges. Since $p$ is an integer, the flow maps
the subintervals of the lower edge onto the subintervals in the upper edge, inducing a permutation $\pi$ of a set of $q$ symbols. 
For example, if $q = 3$ and $p = 1$, then the corresponding permutation 
is $\pi = (012)$ and if $p = 2$ then the corresponding permutation is $\pi = (021)$. It follows that the 
return map $F:P \to P$ can be described as the composition $\mathfrak a \circ R_\pi$ of a permutation of $q$ intervals and the 
von Neumann-Kakutani map. To illustrate this visually, in Figure~\ref{fig:lochness}(a) choose a point on the 
horizontal edge of the square, apply $R_\pi$ obtaining another point on the lower horizontal edge, and then traverse the square 
in the vertical direction. When the vertical flow line reaches the upper horizontal edge of the square, 
apply the von Neumann-Kakutani map $\mathfrak a$ and return to $P$ using the identification \eqref{eq-quot}.

We give the section $P$ a measure $\lambda$ induced from the Lebesgue measure on the interval $I=[0,1)$. 
Since only a countable number of flow lines reach the singularity $\sigma$ in $\overline{L}$, the discussion above 
results in the following statement. Denote by $j: P \to I = [0,1)$ the inclusion map.

\begin{prop}\label{prop-shift}
Let $q \geq 2$, $p \geq 1$ be integers, and let $F: P \to P$ be the Poincar\'e map of the flow $\varphi_{\theta}$, 
where $\theta = \tan^{-1}(q/p)$. Then there exists a permutation $\pi$ of $q$ symbols, such that $j: P \to I$ induces a measure-theoretical isomorphism of the 
dynamical systems $(P,F,\lambda)$ and $(I, F_\pi,\lambda)$, where $F_\pi = \mathfrak a \circ R_\pi$ is a rotated 
odometer and $\lambda$ is the Lebesgue measure.
\end{prop}

\subsection{Rotated odometers and flows} We now prove Theorem~\ref{thm-main0}, that is, we show that any
rotated odometer is measure-theoretically isomorphic to a Poincar\'e map of a flow of rational slope on an infinite-type translation surface of finite area. The topology of this surface may be slightly more complicated than that of $L$, since in order to obtain the given 
permutation $\pi$ we may have to apply additional identifications on the vertical sides. As a result, a finite number of cone angle singularities may arise, which we will have to remove from $L$, creating additional planar ends.

We call the surface with infinite genus and one non-planar and a finite number of planar ends the \emph{Loch Ness monster with whiskers}, see Figure~\ref{fig:LNM2}.

\begin{proofof}{Theorem~\ref{thm-main0}}
To obtain a given permutation of the intervals $\{I_i\}_{i=0}^{q-1}$ we divide the vertical sides of the square $S$ into subintervals of equal length and permute them. For us to be able to do that, 
all flow lines starting at the horizontal edge of $S$ must intersect the vertical side, so the angle $\theta$ must be less than $\frac{\pi}{4}$, so $\tan(\theta) <1$. 
Take any $p=mq+r$ with $m \geq 1$ and $0 \leq r \leq q-1$ and let  $\theta = \tan^{-1}(\frac{q}{p})$.

Consider $m+2$ copies of $S$ such that for $1 \leq i \leq m+1$ the right vertical edge of the $i$-th copy is identified with the left vertical edge of the $i+1$-st copy. Let $\{I_k \}_{k = 0}^{q-1}$ be the subdivision
of the lower horizontal edge into intervals of length $\frac{1}{q}$, 
and $\{J_k\}_{k = 0}^{p-1}$ and $\{J_k'\}_{k = 0}^{p-1}$ be the subdivisions of respectively the left and the right 
vertical edges of the first copy of $S$ into intervals of length $\frac{1}{p}$. 
The numbering of the intervals in the vertical edges increases from bottom to top.

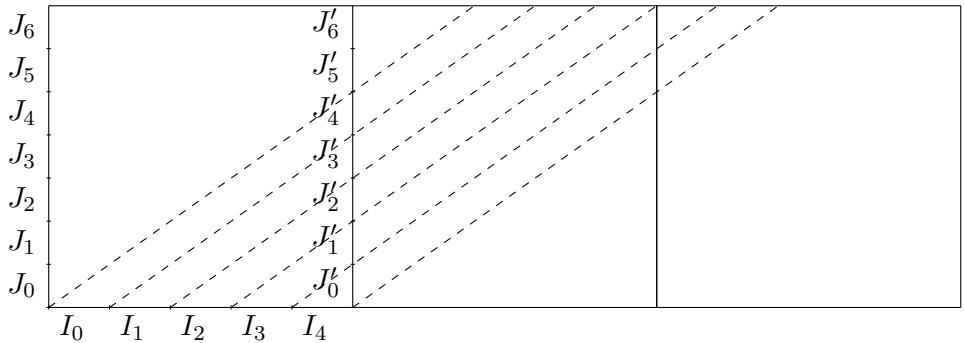
\begin{figure}[ht]
\begin{center}
\begin{tikzpicture}[scale=0.8]
\draw[-, draw=black] (0,5)--(15,5); 
\draw[-, draw=black] (5,5)--(5,0); 
\draw[-, draw=black] (10,5)--(10,0); 
\draw[-, draw=black] (0,5)--(0,0); 
\draw[-, draw=black] (10,5)--(10,0); 
\draw[-, draw=black] (15,5)--(15,0); 
\draw[-, draw=black] (0,0)--(15,0); 
\draw[dashed, draw=black] (0,0)--(7,5); 
\draw[dashed, draw=black] (1,0)--(8,5); 
\draw[dashed, draw=black] (2,0)--(9,5); 
\draw[dashed, draw=black] (3,0)--(10,5); 
\draw[dashed, draw=black] (4,0)--(11,5); 
\draw[dashed, draw=black] (5,0)--(12,5); 
\foreach \y in  {0,1,2,3,4,5,6}
    \draw (1pt,0.715*\y cm) -- (-1pt,0.715*\y cm) node[anchor=south east] {$J_{\y}$};
\foreach \y in {0,1,2,3,4,5,6}
    \draw (5cm+1pt,0.715*\y cm) -- (5cm-1pt,0.715*\y cm) node[anchor=south east] {$J_{\y}'$};
\foreach \x in {0,1,2,3,4}
    \draw (\x cm,-1pt) -- (\x cm,1pt) node[anchor=north west] {$I_{\x}$};

\end{tikzpicture}
\caption{The flow for the rotation by $\frac{7}{5}$, so $q = 5$, $p= 7$ and $r=2$; if the vertical edges are identified by a single translation, like in the torus, then the corresponding permutation of $\{I_k\}_{k = 0}^{4}$ is $\pi=(02413)$. Identifications of the vertical sides using a non-trivial permutation $\pi'$ of $\{J_k'\}_{k=0}^6$ may result in a different $\pi$. When glueing the vertical edges using a non-trivial $\pi'$, we have to remove the endpoints of the intervals in $\{J_k'\}_{k=0}^6$ from $S$, possibly creating planar ends in the resulting surface.}
\label{fig:lowflow}
\end{center}
\end{figure}

Each flow line of $\varphi_\theta$ intersects at least $m+1$ and at most $m+2$ copies of $S$ before reaching the upper horizontal edge. For example,
in Figure~\ref{fig:lowflow} the flow lines of the points in the intervals $I_3$ and $I_4$ intersect three copies of $S$,
while the flow lines of the points in  $I_0,I_1,I_2$
intersect only two copies of $S$.
The flow lines intersect only the first $q$ elements
$\{ J'_k\}_{k=0}^{p-1}$, and so the intersection of the flow lines with the right vertical 
edge of the first copy of $S$ defines a map
  $$s: \{0,1,\ldots, q-1\} \to \{0,1,\ldots, p-1\}, \quad i \mapsto k = q-1-i$$
  with range $\{0,\ldots,q-1\}$. There is the partial inverse $s^{-1}:\{0,1,\ldots,q-1\} \to \{0,1,\ldots,q-1\} $. The intersection of the flow lines with the upper horizontal edge followed by the identification of the copies of $S$ as in the standard torus defines the map
  $$t: \{0,1,\ldots,q-1\} \to \{0,1,\ldots,q-1\}, \quad i \mapsto i+r \bmod q.$$

We set
   $$\pi' =  s \circ t^{-1} \circ \pi \circ  s^{-1}$$ 
and define the identifications of the intervals of the partitions $\{J_k\}_{k=0}^{p-1}$ and $\{J_k'\}_{k=0}^{p-1}$ by
\begin{equation}\label{eq-identlower}
J_k' \sim \begin{cases}
 J_{\pi'(k)}& \textrm{ if } 0\leq k \leq q-1, \\
 J_k & \textrm{ if } q \leq k \leq p-1.
\end{cases}
\end{equation}
In words, we permute the lowest $q$ intervals in the partition $\{J_k\}_{k=0}^{p-1}$ of the vertical side of $S$, and we keep the top $p-q$ intervals not permuted.

Build a surface $L_{\pi,p}$ as in Section~\ref{subsec:LNM}, but identify the vertical sides of $S$ using \eqref{eq-identlower}. Then flowing from the lower to the upper horizontal edge in $S$ produces the permutation $\pi$ on the intervals 
of the subdivision of the horizontal edge. Indeed, flowing from the horizontal to the vertical edge maps each $I_k$, $k=0,\ldots q-1$, to the set $s(I_k)$ of the partition of the vertical edge, essentially re-numbering the sets in the partition $\{I_k\}_{k=0}^{q-1}$ in the opposite order. We apply $\pi'$, and then the partial inverse $s^{-1}$, which reverses the re-numbering of the sets, and incorporates $\pi'$ into the partition $\{I_k\}_{k=0}^{q-1}$. We then apply the map $t$, which implements the translation by $p/q$ on a square with identified vertical edges. Thus
  $$t \circ s^{-1} \circ \pi' \circ s = t \circ s^{-1} \circ (s \circ t^{-1} \circ \pi \circ  s^{-1}) \circ s = \pi.$$

It follows that 
the Poincar\'e map $F:P \to P$ of this flow is measure-theoretically isomorphic to the rotated odometer $(I,F_\pi)$. 

It remains to check that the surface $L_{\pi,p}$ which is obtained from $L$ by applying the
identifications of subintervals on vertical edges is still a translation surface with a wild singularity. Some of the 
upper endpoints of the intervals $J_0,\ldots, J_{p-2}$ may be identified with the singularity $\sigma$, 
the remaining ones are identified with each other. Since there is a finite number of them, this results in at most 
a finite number of additional cone angle singularities of the surface $L$, which are isolated from $\sigma$. The number of saddle connections and the number of 
non-separating curves that the surface can admit remains infinite, 
so the resulting surface is a translation surface with a wild singularity and of infinite genus. 
By \cite[Proposition 3.10]{Rbook} the number of ends of $L_{\pi,p}$ 
is in one-to-one correspondence with the set of singularities of $L_{\pi,p}$, and by \cite[Proposition 3.6]{Rbook} every cone angle singularity corresponds to a planar end. 
Therefore, $L_{\pi,p}$ is a Loch Ness monster with ``whiskers'', see Figure~\ref{fig:LNM2}, 
where the number of whiskers corresponds to the number of cone angle singularities. 
\end{proofof}

\begin{example}\label{eq-whiskers}
{\rm
Let $q = 5$, and $\pi = (0)(1)(23)(4)$. Let $p=q$ then $\theta = \pi/4$, $r=0$ and so $t$ is the trivial permutation. Applying Theorem~\ref{thm-main0} we obtain $\pi' = (0)(12)(3)(4)$. Considering the identification of vertical edges of the unit square given by $\pi'$, one obtains that the metric completion of the surface $L_{\pi,p}$ has one wild singularity, one cone angle singularity of multiplicity $3$ and one removable cone angle singularity
(i.e., of multiplicity $1$). Thus $L_{\pi,p}$ has one non-planar and one planar end (we are not counting the removable singularity).
}
\end{example}

\section{Periodic points and the unique minimal set}\label{sec:general}

In this section we prove Theorem~\ref{thm-main1} which gives a basic description of the dynamics of periodic and non-periodic points of the rotated odometer $(I,F_\pi)$, 
where $F_\pi = \mathfrak a \circ R_\pi$ is as in \eqref{eq-rotod}. 
Given an integer $q \geq 2$ and a permutation $\pi$ of $q$ symbols, $R_\pi: I \to I$ is a finite IET of $q$ subintervals of $I$ of equal length, determined by $\pi$. 

\begin{lemma}\label{lemma-invert}
The map $F_\pi = \mathfrak a \circ R_\pi:I \to I$ is invertible at every point in $I$ except $0$.
\end{lemma}

\begin{proof}
The statement follows from the fact that the range of $\mathfrak a$ is $(0,1)$, the range of $R_\pi$ is $[0,1)$, and both $\mathfrak a$ and $R_\pi$ are translations on their intervals of continuity.
\end{proof}

We introduce a few notations which we use throughout the paper.

\begin{definition}\label{defn-not}
Define
$$
N = \min \{n \in \mathbb{N} : 2^n >q\} \quad \textrm{and} \quad L_k = [0,2^{-kN}).
$$
  For fixed $q$ and $N$, and any $k \geq 0$, we denote the partition of $I$ into $q2^{kN}$ half-open subintervals of equal length by
  $$\mathcal{P}_{kN,q} = \left\{I_{k,i}: = \left[\frac{i}{q2^{kN}}, \frac{i+1}{q2^{kN}} \right) : 0 \leq i \leq q 2^{kN}-1 \right\}.$$
\end{definition}
In particular, $\mathcal{P}_{0,q}$ is the partition of $I$ into $q$ subintervals, and $I_{0,i} = I_i$, $i=0,\ldots,q-1$, where $I_i$ are the subintervals in Section~\ref{sec:flow}.

The discontinuities set of $F_\pi = \mathfrak a \circ R_\pi$ is
 \begin{align}\label{eq-setd0}D_0 = \{R_\pi^{-1}(1 -2^{-n}) : n \geq 0 \}.\end{align}
Denote by $D^+$ the set of forward orbits of the points in $D_0 $, and by $D^-$ the set of backward orbits 
of $D_0  \setminus \{0\}$. 
Set $D = D_0 \cup D^+ \cup D^-$. 
\begin{lemma}\label{lemma-rational}
The set $D $ is contained in the set
$\left\{ \frac{p}{q2^{n}} : n \in \mathbb{N}, 0 \leq p \leq q2^n-1\right\}$.
\end{lemma}

\begin{proof} Note that the restriction of $R_\pi$ to each interval in $\mathcal{P}_{0,q}$ is a translation
by a rational number with denominator $q$, and  $\mathfrak a$ acts by translations 
by multiples of $2^{-n}$, $n \geq 1$, with $n$ depending on a point in $I$.  
\end{proof}

For a point $x \in I$, let $\operatorname{\text{orb}}^+(x)$, $\operatorname{\text{orb}}^-(x)$ and $\operatorname{\text{orb}}(x)$ be the forward, backward and two-sided orbit of $x$ under $F_\pi$ respectively.
Clearly if $\operatorname{\text{orb}}^+(x)$ is periodic, then $F_\pi$ is invertible at every point of $\operatorname{\text{orb}}^+(x)$ and so $x$ has a two-sided periodic orbit.

As in the Introduction, we denote by $I_{per}$ the set of points in $I$ whose orbits under $F_\pi$ are periodic, and by $I_{np}$ the complement of $I_{per}$ in $I$.

\begin{prop}\label{prop:accum}
Consider the dynamical system $(I,F_\pi)$. We have the following.
\begin{enumerate}
\item \label{it1-1} For every $x \in I$, the forward orbit $\operatorname{\text{orb}}^+(x)$
is either periodic or accumulates at $0$.
\item \label{it1-2} The system $(I_{np},F_\pi)$ has a unique minimal set, denoted by $I_{min}$.
\item \label{it1-3} If non-empty, the set of periodic points $I_{per}$ is at most a countable union of half-open intervals with left
and right endpoints in $D$.
\end{enumerate}
\end{prop}

\begin{proof} 
Let $x \in I$ and suppose $\operatorname{\text{orb}}^+(x) \cap L_k = \emptyset$ for some $k \geq 1$. We show that $\operatorname{\text{orb}}(x)$ is periodic.

Consider the partition $\mathcal{P}_{kN,q}$ of $I$ from Definition~\ref{defn-not}.
Let $L_{k}' = R^{-1}_\pi([1 - 2^{-kN},1))$
 and note that $\operatorname{\text{orb}}^+(x)$ visits $L_k$ if and only if it visits $L_k'$, since $\mathfrak a$ maps the interval $[1 - 2^{-kN},1)$ into $L_k$. 
 Therefore, since by assumption $\operatorname{\text{orb}}^+(x)$ does not visit $L_k$, then it is contained in $I \setminus L_k \cup L_k'$. Note that the discontinuities of $F_\pi$ in $I \setminus L_k \cup L_k'$ can only be at the left endpoints of the sets in $\mathcal{P}_{kN,q}$, and denote by
  $$\mathcal{P}'_k = \{I_j \in \mathcal{P}_{kN,q} : I_j \cap \operatorname{\text{orb}}^+(x) \ne \emptyset\}$$
the collection of sets in $\mathcal{P}_{kN,q}$ visited by the orbit of $x$.   
Then the restriction $F_{\pi}|_{I_j}$ is continuous for each $I_j \in \mathcal{P}'_k$.
Thus $F_{\pi}$ permutes the intervals in $\mathcal{P}'_k$,
and since there are only finitely many of them and $F_\pi$ is injective, the orbit of $x$ and so of each $I_j \in \mathcal{P}_k'$ is periodic. It follows that
$I_{per}$
is the union of half-open intervals.

If $\operatorname{\text{orb}}^+(x)$ is not periodic, then it must visit every $L_k$, $k \geq 1$, so $\operatorname{\text{orb}}^+(x)$ accumulates at $0$. This proves parts \eqref{it1-1} and \eqref{it1-2}.

We have seen in the proof above that $I_{per}$ is the union of a subcollection of intervals in the partitions $\mathcal{P}_{kN,q}$, for $k \geq 1$. We now claim that these intervals can assemble into larger half-open intervals of periodic points, so that the endpoints of these intervals are in $D$. Indeed, suppose an endpoint  $x$ is not in $D$. Then $F_\pi$ is continuous at every point in $\operatorname{\text{orb}}(x)$, and since $F_\pi$ is a translation on its intervals of continuity, $x$ has an open neighborhood which consists of points periodic with the same period as $x$. This proves part \eqref{it1-3}.
\end{proof}

Examples~\ref{ex-countableperiodic} and \ref{ex-finiteperiodic} are rotated odometers with respectively countable and finite numbers of non-empty intervals of periodic points.

\begin{proofof}{Theorem~\ref{thm-main1}}
Follows from Lemma~\ref{lemma-invert} and Proposition~\ref{prop:accum}.
\end{proofof}

\begin{example}\label{ex-distortedodometer}
{\rm
We note that the property of having an infinite collection of intervals of periodic points is also exhibited by infinite IETs which are not rotated odometers. For instance, in~\cite{HRR} half-open subintervals of $I$ are rearranged in the manner similar to the von Neumann-Kakutani map,
but the lengths of the subintervals are different than the lengths in \eqref{eq-odometer}. Namely, $b:I \to I$ is given by
$$
b(x) = x - 1 + k^{-1} + (k+1)^{-1} \quad \textrm{ for } 1 - k^{-1} \leq x  < 1- (k+1)^{-1}, \, k \in \mathbb{N}.
$$
In this system, $I_{per}$ is an infinite union of intervals where each point is periodic, and the complement of $I_{per}$ is minimal.
}
\end{example}

\section{Compactification to a Cantor system}\label{sec-compactification}

In this section, we compactify the rotated odometer $(I,F_\pi)$ to a dynamical system $(I^*,F_\pi^*)$ given by a homeomorphism $F_\pi^*$ of a Cantor set $I^*$, where $I=[0,1)$ and $F_\pi$ is defined by \eqref{eq-rotod}. The goal of this procedure is to get rid of discontinuities, and to do that we employ the well-known procedure of doubling the discontinuity points. Since $I^*$ is obtained by adding to $I$ a countable number of points, $(I_{np},F_\pi)$ and $(I_{np}^*,F_\pi^*)$ are measurably isomorphic. We summarize the results of this section in the following theorem.

\begin{theorem}\label{thm-compactify}
Let $\pi$ be a permutation of $q \geq 2$ symbols, and let $(I,F_\pi)$ be a rotated odometer. 
Let $I = I_{per} \cup I_{np}$ be a decomposition of $I$ into the sets of periodic and aperiodic points. 

Then there exists an inclusion $\iota: I \to I^*$ into a compact space $I^* = I_{per}^* \cup I_{np}^*$, and a homeomorphism $F_\pi^*: I^* \to I^*$ with the following properties:
\begin{enumerate}
\item The complement of $I$ in $I^*$ is countable, and $\iota(I_{per})$ and $\iota(I_{np})$ are contained in $I_{per}^*$ and $I_{np}^*$ respectively.
\item $F_\pi^* \circ \iota = \iota \circ F_\pi$.
\item $I_{np}^*$ is a Cantor set, and every point in $I_{np}^*$ is aperiodic under $F_\pi^*$.
\item There exists a measure $\mu$ on $I^*$ such that $(I,F_\pi,\lambda)$, where $\lambda$ is Lebesgue measure, and $(I^*,F_\pi^*,\mu)$ are measurably isomorphic via the map $\iota$.
\end{enumerate}
\end{theorem}

\subsection{Doubling the points.}
Recall from Section~\ref{sec:general} that we introduced the set $D = D_0 \cup D^- \cup D^+$, where $D_0$, defined by \eqref{eq-setd0}, is the set of discontinuities of the map $F_\pi$, and $D^+$ and $D^-$ are the sets of forward and backward orbits of the points in $D_0$ respectively. For each point in $D$ we add its double by setting
  $$D^* = \{x^- : x \in D \setminus \{0\}\} \cup \{1\}.$$
  Next, we consider the accumulation point of the discontinuity points. 
  In the von Neumann-Kakutani map, the discontinuity points accumulate at $1$. In the rotated odometer, 
  the accumulation point is shifted by the inverse of $R_\pi$, so we denote
 \begin{align}\label{eq-xhat}
 \widehat x := \lim_{y \nearrow 1} R_{\pi}^{-1}(y). 
 \end{align}
The point $\widehat x$ is either in the interior of $I$, or $\widehat x = 1$. The latter happens if and only if $R_\pi$ fixes the last interval in the partition $\mathcal{P}_{0,q}$. If $\widehat x$ is in the interior of $I$, then $\widehat x$ is the double of one of the points in $\{R_\pi^{-1}(1 -2^{-n}) : n \geq 0 \}$. In both cases $\widehat x$ is in $D^*$.

Let $I^* = I \cup D^*$. To underline that the points in $D$ are left endpoints of intervals of continuity,  for all $x \in D$ denote $x^+:= x \in I^*$. The points in $D^*$ are the added right limit points. If $\widehat x = 1$, then set $\widehat x^- = \widehat x= 1$, and $\widehat x^+ = 0$.

\subsection{Order topology on $I^*$.} 
The interval $I \cup \{1\}$ has order $ \leq $ induced from $\mathbb{R}$. We extend this order to an order on $I^* = I \cup D^*$ by defining: 
\begin{enumerate}
\item $x^- \leq x^+$ for all $x \in D$, 
\item for all $y \in I \cup \{1\}$ and all $x \in D$, if $y \leq x$ then $y \leq x^-$. 
\end{enumerate}
For all $x^+ \in D$ there are no points between $x^-$ and $x^+$, so adding $x^-$ to $I$ can be thought of as creating a gap.
We equip $I^*$ with the order topology with open sets given by
 $$\mathcal{B} = \{(a,b) : a,b \in I^*\} \bigsqcup \{[0,b) : b \in I^*\} \bigsqcup \{(a,1] : a \in I^*\}.
 $$
\begin{lemma}\label{lemma-closedopen}
Let $x^+ \in D$ and $y^- \in D^*$ with $x^+ < y^-$. Then the subset $[x^+,y^-]$ of $ I^*$ is closed and open.
\end{lemma}

\subsection{Extension to a homeomorphism of $I^*$}
We define an extension $F_\pi^*:I^* \to I^*$ to coincide with $F_\pi$ on the points in $I$, and so by definition $F_\pi^*: D \to D$. It remains to define $F_\pi^*$ on $D^*$. 

For every $x^- \in D^*\setminus \{1\}$ there is $x^+ \in D$, such that $x^- = \lim_{y \nearrow x^+} y$. If $x^- \ne \widehat x^-$, that is, if $x^-$ is not the accumulation point of discontinuities $\widehat x^-$, then $x^+$ has a half-neighborhood on the left where $F^*_\pi$ is continuous, and we define 
$$F_\pi^*(x^-) = \lim_{y \nearrow x^+} F_\pi(y).$$ 
We also set $F_{\pi}^*(\widehat x^-) = 0$, so we have for the forward orbit of $\widehat x^-$
  $${\rm orb}^+(\widehat x^-) = \{\widehat x^-\} \cup {\rm orb}^+(0).$$ 
 If $\widehat x^- \ne 1$, then also set 
  $$F_\pi(1) = \lim_{y \to 1} F_\pi(y).$$ 
The above remarks are summarized in the following statement.

\begin{prop}\label{lemma-extbijection}
For the map $F_\pi^*:I^* \to I^*$ the following holds:
\begin{enumerate}
\item \label{item-1} $F_\pi^*$ maps points of $D^* \setminus \{\widehat x^-\}$ to $D^*$, and points of $D$ to $D$.
\item \label{item-3} $F_\pi^*$ is a bijection, and $(F_\pi^*)^{-1}(D) \subset D$.
\item $\widehat x^-$ is aperiodic under $F_\pi^*$.
 \end{enumerate}
\end{prop}

\subsection{Periodic and non-periodic points of the compactified system}
By Theorem~\ref{thm-main1} we have the decomposition $I = I_{per} \cup I_{np}$ 
of the unit interval into the $F_\pi$-invariant sets of 
periodic and non-periodic points respectively, and $I_{per}$ is a countable (possibly empty) union of half-open intervals with endpoints in $D$.

Let $[x,y) \subset I_{per}$ with $x,y \in D$. Then there is $y^- \in D^*$ corresponding to $y$, and $y^-$ is periodic under the extension of $F_\pi$ to $I^*$. Define 
  \begin{align}\label{eq-per}I_{per}^* = \{[x^+,y^-] \subset I^* : [x,y) \subset I_{per},\,  x,y \in D\}.\end{align}

The following lemma is a direct consequence of the definitions.

\begin{lemma}\label{lemma-inpclosure}
 We have the following properties:
\begin{enumerate}
\item \label{it1-6} $I_{per}$ is an open subset of $I_{per}^*$. 
\item \label{it1-7} $I_{per}^*$ is an open subset of $I^*$.
\item $I_{per}^*$ is invariant under the map $F_\pi^*$.
\end{enumerate}
\end{lemma}

Let  $I_{np}^* = I^* \setminus I_{per}^* $ be the complement, so $I^* = I^*_{np} \cup I^*_{per}$. Clearly $I_{np} \subset I_{np}^*$.

\begin{prop}\label{prop-inpclosure}
 We have the following properties:
\begin{enumerate}
\item \label{it1-8} $I_{np}^*$ consists of aperiodic points.
\item \label{it1-5} $I^*$ is compact, and $I^*_{np}$ is closed in $I^*$. Moreover, $I_{np}^* = \overline{I_{np}}$, the topological closure of $I_{np}$ in $ I^*$.
\item  \label{it1-85} $I_{np}^*$ is invariant under the map $F_\pi^*$.
\end{enumerate}
\end{prop}

\begin{proof} To show \eqref{it1-8}, note that if $z \in I^* \setminus I_{per}$ is periodic, then it must be in $D^*$. Then $z \ne \widehat x^-$ and must be the right endpoint of an interval of continuity. Since $F_\pi$ is a translation on its intervals of continuity, $z$ is the right endpoint of an interval of periodic points, and so $z \in I_{per}^*$, and \eqref{it1-8} and \eqref{it1-85} follow.
To show item \eqref{it1-5}, we note that $I^*$ is a totally ordered set with order topology, so it is compact, and $I_{np}^*$ is closed since $I_{per}^*$ is open.
\end{proof}

\subsection{Properties of the aperiodic subsystem $(I_{np}^*,F_\pi^*)$} 

\begin{prop}\label{cor-disjoint}
We have the following.
\begin{enumerate}
\item The set $D$ 
is dense in $I_{np}$.
\item The set $I_{np}^*$ is a Cantor set. 
\item The restriction $F_\pi^*: I^*_{np} \to I^*_{np}$ is a homeomorphism.
\end{enumerate}
\end{prop}

\begin{proof} We have to show that every non-empty intersection $(x,y) \cap I_{np}$ contains a point of $D$.
If $(x,y)$ contains an interval of periodic points, then obviously $(x,y) \cap I_{np}$ contains a discontinuity, 
so assume that there are no periodic points in $(x,y)$.

By item \eqref{it1-1} of Proposition~\ref{prop:accum} every orbit that is not periodic accumulates at $0$. 
If $F_\pi^n$ is continuous on $(x,y)$ for all $n \geq 1$, then the orbit of every point in $(x,y)$
gets arbitrarily close to $0$, which is impossible since $F_\pi^n$ is a translation and so it preserves distances 
between the points in $(x,y)$. Therefore, $F^n_\pi(x,y)$ must contain a discontinuity for some $n \geq 1$, and $D$ is dense in $I_{np}$. 
It follows by standard arguments that $I_{np}^*$ is a Cantor set 
and $F_\pi: I^*_{np} \to I^*_{np}$ is a homeomorphism. 
\end{proof}

Define a measure $\mu$ on $I^*$ by setting $\mu([x^+,y^-]) = y -x$ for every clopen set 
$[x^+,y^-] \subset I^*$, $x,y \in D$. Since $D^*$ is countable, the following conclusion is straightforward.

\begin{lemma}\label{lemma-isomCantor}
The inclusion $\iota: I \to I^*$ induces a measurable isomorphism of dynamical systems $(I,F_\pi,\lambda)$ and $(I^*,F^*_\pi,\mu)$, where $\lambda$ is Lebesgue measure.
\end{lemma}

\begin{remark}
{\rm
Since the map $F_\pi:I \to I$ is not invertible at $0$, the set 
$I_{np}$ contains $\operatorname{\text{orb}}^+(0)$. A natural question for which we do not know the answer is, if it is possible that the orbits of all discontinuity points in $D_0$, except for the orbit of $0$, are periodic.
}
\end{remark}

\section{Bratteli-Vershik systems}\label{sec:BV}

In this section we construct a Bratteli-Vershik system which is 
conjugate to the aperiodic system $(I^*_{np},F^*_\pi)$ defined in Section~\ref{sec-compactification}. We start by recalling the background on Bratteli-Vershik systems in Sections~\ref{subsec:subst} and \ref{subsec:BV}. 
The main technical result is stated in Theorem~\ref{thm-main3l}, which is then proved in Sections~\ref{subsec:firstreturn} - ~\ref{subsec:conjrotated}.

\subsection{Substitutions}\label{subsec:subst}

We recall some standard constructions in symbolic dynamics, a good reference for which is the survey by Durand \cite{Dur10}. 

Let $\mathcal{A} = \{0, \dots, q-1\}$ be a finite alphabet, let $\mathcal{A}^*$ be the set of all words of finite length 
in this alphabet, and let $\Sigma = \mathcal{A}^{\mathbb{N}}$ be the set of infinite sequences in this alphabet.

\begin{definition}\label{def:substitution}
A {\em substitution} $\chi: \mathcal{A} \to \mathcal{A}^*$ is a map which assigns to every $a \in \mathcal{A}$ a single word 
$\chi(a) \in \mathcal{A}^*$, and which
extends to $\mathcal{A}^*$ and $\Sigma$ by concatenation:
$$
\chi(b_1b_2\dots b_r) = \chi(b_1)\chi(b_2)\dots \chi(b_r), \, r \geq 1.
$$
The $q\times q$ matrix $M$, where the $i,j$-th entry is the number of letters $j$
in $\chi(i)$, is called the \emph{associated matrix} of $\chi$.
\end{definition}

\begin{definition}\label{def:primproper}
We say that a substitution $\chi: \mathcal{A} \to \mathcal{A}^*$ is: 
\begin{itemize}
\item {\em primitive}, if there is $r \geq 1$ such that for all $i\in \mathcal{A}$, 
$\chi^r(i)$ contains every letter in $\mathcal{A}$.
\item {\em proper} if there exist two letters $a,b \in \{ 0, \dots, q-1\}$
such that for all $i \in \{ 0,\dots, q-1\}$, the first letter of $\chi(i)$ is $a$
and the last letter of $\chi(i)$ is $b$. 
\end{itemize}
\end{definition}

The substitutions we consider will sometimes be \emph{primitive}, and they will always be \emph{proper}.
  
If $\chi(a)$ starts with $a$, we get a fixed point of $\chi$ which (unless $\chi(a) = a$) is an infinite sequence
\begin{align}\label{eq-rho}
\rho = \rho_1\rho_2\rho_3 \dots = \lim_{j \to \infty} \chi^j(a) \in \Sigma.
\end{align}
For sequences $s = (s_k) \in \Sigma$, define the \emph{left shift} by
\begin{align}\label{eq-shiftmap}\sigma: \Sigma \to \Sigma, \qquad  s_1 s_2 \ldots \mapsto s_2 s_3 \ldots. \end{align}
Consider the topological closure $X_\rho = \overline{\operatorname{\text{orb}}_\sigma(\rho)}$, where $\operatorname{\text{orb}}_\sigma$ is the orbit of $\rho$ under $\sigma$. The dynamical system $(X_\rho,\sigma)$ is called a \emph{subshift}.

\begin{definition}\label{def:linrecurr}
The subshift $(X_\rho,\sigma)$ is \emph{linearly recurrent} if there is  $L\geq 1$ such that
for every sequence $s \in X_\rho$ and $k \geq 1$, every finite word $w = w_1 \ldots w_k$ in $s$ reoccurs within $L|w|$ entries,
where $|w|$ denotes the length of $w$.
\end{definition}

If the substitution $\chi$ is primitive, then $(X_\rho,\sigma)$ is linearly recurrent and minimal, see for instance \cite{Dur10}. Then we can use, for instance, the results of \cite{CDHM03} to compute eigenvalues and eigenfunctions of the Koopman operator of this dynamical system.

\subsection{Bratteli-Vershik systems}\label{subsec:BV}
We now define Bratteli diagrams and Bratteli-Vershik systems.

\begin{definition}\label{def:brat}
A {\em Bratteli diagram} $(V,E)$ is an infinite graph with the set of vertices $V = \bigsqcup_{k \geq 0} V_k$ and the set of edges $E = \bigsqcup_{k \geq 0} E_k$ with the following properties:
\begin{itemize}
\item  $V_0$ consists of a single vertex $v_0$, called the \emph{root} of the Bratteli diagram.
\item For $k \geq 0$, $V_k$ is a finite set. 
\item For $k \geq 0$, 
each edge $e \in E_k$ connects the vertex $\boldsymbol{s}(e) \in V_{k}$ to the vertex $\boldsymbol{t}(e)
\in V_{k+1}$, where $\boldsymbol{s},\boldsymbol{t}: E \to V$ are called the {\em source} and the {\em target} maps respectively.
\end{itemize}
In addition, the Bratteli diagram $(V,E,<)$ is \emph{ordered} if for each $k \geq 2$ and $v \in V_k$, there is 
a total order $<$ on the incoming edges $e$ with $\boldsymbol{t}(e) = v$.
\end{definition}

We assume that every $v \in V_1$ is connected to the root $v_0$ by a single edge, so $\# E_0 = \# V_1$. 
We assume that for every $v \in V_{k}$ there exists at least one outgoing edge $e \in E_k$
with $v = \boldsymbol{s}(e)$, and for every $v \in V_{k+1}$ there exists at least one 
incoming edge $e \in E_k$ with $v = \boldsymbol{t}(e)$, for $k \geq 1$.

Recall that a square matrix $M$ is \emph{primitive} if it has a power with strictly positive entries.

\begin{definition}\label{def:associatedmatrixM}
Let $(V,E)$ be a Bratteli diagram. For $k \geq 1$, the {\em associated matrix} $M_k$ to $E_k$  
has $i,j$-entries equal to the number of edges from $j \in V_{k}$ to $i \in V_{k+1}$. 

A Bratteli diagram $(V,E)$ is \emph{stationary} if for all $k \geq 1$ the associated matrices satisfy $M_k = M_1$. 
A stationary Bratteli diagram $(V,E)$ is \emph{simple} if the associated matrix is primitive. 
\end{definition}

Here is an example which is fundamental for the rest of the paper.

\begin{example}\label{example-mainmethod}
{\rm Suppose we are given a substitution $\chi$ with the alphabet $\mathcal{A}$ as in Definition~\ref{def:substitution}.
We construct a Bratteli diagram $(V,E)$ by taking $V_k = \mathcal{A}$, and defining $E_k$ so that there is an edge from $j \in V_{k}$ to $i \in V_{k+1}$ for each appearance of
the letter $j$ in $\chi(i)$, for $k \geq 1$. Then $M_k = M_1$ for all $k \geq 1$, and $(V,E)$ is stationary. If the substitution $\chi$ is primitive, then $(V,E)$ is simple.

We define the order $<$ on the incoming edges to $i \in V_{k+1}$ as the order of the corresponding letters in the word $\chi(i)$.
}
\end{example}

Bratteli diagrams emerged in the area of $C^*$-algebras \cite{Brat72},
and they were given a dynamical interpretation, when Vershik equipped them 
with an order and a successor map, described below and now called the Vershik map \cite{Ver81}. 
It was shown in \cite{HPS1992} that every minimal homeomorphism on the Cantor set 
can be represented as a Bratteli-Vershik system. 
Later Medynets \cite{Med06} extended this to all aperiodic homeomorphisms on the Cantor set. For a general survey, we refer to \cite{Dur10}.

\begin{definition}\label{def:pathBV}
A finite (resp.\ infinite) \emph{path} in the Bratteli diagram $(V,E)$ is a finite (resp.\ infinite) sequence of
edges $(e_k)_{k=0}^m$ (resp. $(e_k)_{k \geq 0}$), such that for all $1 \leq k \leq m$
(resp.\ $k \geq 1$) we have $e_k \in E_k$ and $\boldsymbol{s}(e_k) = \boldsymbol{t}(e_{k-1})$.
\end{definition}

\begin{definition}\label{def:height}
Let $v_0 \in V_0$ be the root of the Bratteli diagram. For $k \geq 1$, define the {\em height} of the vertex $i \in V_k$ 
as the number of finite paths from $v_0$ to $i$:
 $$
  h^{(k)}_i = \#\{ e_0\dots e_{k-1} : \boldsymbol{s}(e_0) = v_0,\ \boldsymbol{t}(e_{k-1}) = i \in V_{k}\}.
 $$
Let $h^{(k)}$ be the vector with entries $h^{(k)}_i$ for $i \in V_k$.
\end{definition}

We define the space of infinite paths of the Bratteli diagram $(V,E)$ by
$$
X_{(V,E)} = \{ (e_k)_{k \geq 0} : e_k \in E_k, \, \boldsymbol{t}(e_k) = \boldsymbol{s}(e_{k+1}) \text{ for all } k \geq 0\},
$$
and, to make it a topological space, we give each finite edge set $E_k$ discrete topology, and equip the space $X_{(V,E,<)}$ 
with the product topology. 

For each $v \in V_{k+1}$, $k \geq 1$, there is a total order $<$ on the set of edges 
$e \in E_k$ such that $\boldsymbol{t}(e) = v$, and so this set has unique minimal and maximal edges. 
If $E_k$ contains a single incoming edge to a vertex $i$ in $V_{k+1}$, then this edge is both minimal and maximal.
We extend $<$ to an order on the paths space $X_{(V,E)}$, by
setting $e < e'$ if there is a minimal $k \geq 0$ such that $\boldsymbol{t}(e_k) = \boldsymbol{t}(e'_k)$ and
$e_k < e'_k$. This turns the path space into a partially ordered space $X_{(V,E,<)}$, 
because we only compare $e$ and $e'$ if $\boldsymbol{t}(e_k) = \boldsymbol{t}(e'_k)$ for some $k \geq 0$.
Let $X_{(V,E,<)}^{\min}$ (resp. $X_{(V,E,<)}^{\max}$) be the subsets of $X_{(V,E,<)}$ 
consisting of paths with only minimal (resp. only maximal) edges.

We now define the {\em Vershik map} $\tau:X_{(V,E,<)} \to X_{(V,E,<)}$.
Given a path $e = (e_k)_{k \geq 0} \in X_{(V,E,<)}$, let $k \geq 0$ be the smallest index such that $e_k \in E_k$
is not the maximal incoming edge at $v_k \in V_k$ with respect to $<$. 
Then put 
$$
\begin{cases}
\tau(e)_j = e_j \quad \text{ for  } j > k, \\
\tau(e)_k \text{ is the successor of } e_j \text{ among all incoming edges at $v_k$,}\\
\tau(e)_0 \dots \tau(e)_{k-1} \text{ is the minimal path
connecting } v_0 \text{ with } \boldsymbol{s}(\tau(e)_k).              
\end{cases}
$$
If no such $k$ exists, then $e \in X_{(V,E,<)}^{\max}$, and we have to choose
$e' \in X_{(V,E,<)}^{\min}$ to define $\tau(e) = e'$. For the rest of the paper we assume that there is a unique minimal sequence $e^{\min}$ and
a unique maximal sequence $e^{\max}$, so we can define $\tau(e^{\max}) = e^{\min}$
which makes $\tau$ into a homeomorphism.

\begin{definition}\label{def-BVsystem}
Let $(V,E,<)$ be an ordered Bratteli diagram with unique maximal and unique minimal paths. This diagram together with the Vershik map $\tau:X_{(V,E,<)} \to X_{(V,E,<)}$, defined as above, is called the \emph{Bratteli-Vershik system}.
\end{definition}

\subsection{Main result of the section}\label{subsec:mainresultsection}

Recall that a sequence $(\chi_k)_{k \geq 1}$ is \emph{pre-periodic} if there exists $k_0 \geq 1$ and $p_0 \geq 1$ such that $\chi_{k} = \chi_{k + p_0}$ for all $k \geq k_0$.

\begin{theorem}\label{thm-main3l}
Consider the aperiodic Cantor system $(I_{np}^*, F^*_\pi)$, and let $\mathcal{A}=\{0,1,\ldots,q-1\}$ be a finite alphabet. There exists a sequence $(\chi_k)_{k \geq 1}$ of substitutions 
   $$\chi_k: \mathcal{A} \to \mathcal{A}^*, \quad i \mapsto \chi_k(i),$$
and an ordered Bratteli diagram $(V,E,<)$ with the following properties:
\begin{enumerate}
\item \label{it-main3-1} The sequence $(\chi_k)_{k \geq 1}$ is pre-periodic.
\item The set $V_0 = \{v_0\}$ is a singleton, and for $k \geq 1$ the vertex set $V_k$ is identified with a non-empty subset of $\mathcal{A}$. 
\item The edge set $E_k$ and the order on the subset of incoming edges for $i \in V_k$ is determined by the substitution $\chi_k$. 
\item The path space $X_{(V,E,<)}$ of the diagram $(V,E,<)$ has a unique maximal and a unique minimal path, and the
Vershik map $\tau: X_{(V,E,<)} \to X_{(V,E,<)}$ is a homeomorphism.
\item There is a homeomorphism $\psi: (I_{np}^*, F^*_\pi) \to (X_{(V,E,<)},\tau)$, such that $\psi \circ F^*_\pi = \tau \circ \psi$.
\end{enumerate}
\end{theorem}

In the rest of this section we prove Theorem~\ref{thm-main3l}. 

\subsection{First return maps}\label{subsec:firstreturn}
Given $q \in \mathbb{N}$, recall that by Definition~\ref{defn-not} 
  $$N = \min \{n \in \mathbb{N} : 2^n \geq q\} \textrm{ and } L_k = [0, 2^{-kN}) \textrm{ for } k \geq 1.$$
Our analysis of the infinite IET $F_\pi = \mathfrak a \circ R_\pi$ is by means of the successive first return maps
to the sections $L_k$. Denote by 
\begin{align}\label{eq-codpart}\mathcal{P}^{cod}_{kN,q} = \{I_{k,i} : 0 \leq i \leq q-1\}\end{align}
the partition of $L_k$ into $q$ intervals of equal 
lengths of the partition $\mathcal{P}_{kN,q}$ given by Definition~\ref{defn-not}.

Recall from the introduction that $F_\pi = \mathfrak a \circ R_\pi$, where $\pi$ is a permutation of $q$ symbols and 
$R_\pi: I \to I$ is an IET with finite number of intervals of equal length, induced by $\pi$. 
We will prove that for $k \geq 1$, the return map $F_{\pi,k}$ has a similar property, as described by the following proposition.

\begin{prop}\label{lem:periodicpermutations}
Let $F_{\pi,k}: L_k \to L_k$, $k \geq 1$ be the first return maps. 
Then there exist permutations $\pi_k$ of $q$ symbols, and finite IETs $R_{\pi,k}:L_k \to L_k$ of the partitions $\mathcal{P}_{kN,q}^{cod}$, defined by $\pi_k$, such that 

\begin{enumerate}
\item \label{it1-nn} $F_{\pi,k} = \mathfrak a_k \circ R_{\pi,k}$, where $\mathfrak a_k$ is a scaled copy of the von Neumann-Kakutani map $\mathfrak a$, given by
$$\mathfrak a_k(x) = \frac{1}{2^{kN}}\, \mathfrak a \left(2^{kN}(x) \right).$$
\item \label{it2-nn} The sequence $(\pi_k)_{k \geq 1}$ is pre-periodic.
\item If $(\pi_k)_{k \geq 1}$ is strictly pre-periodic, and $S = \{\pi_1,\ldots, \pi_{k_0-1}\}$ is the pre-periodic part of the sequence, then none of the permutations in $S$ occurs in the periodic part.
\end{enumerate}
\end{prop}

\begin{proof} 
We argue by induction. 

For $k= 0$ we have $F_{\pi,0} = F_\pi$, $R_{\pi,0} = R_\pi$ and $\mathfrak a_0 = \mathfrak a$ by definition. 

For the induction step, assume that the statement of the proposition holds for $F_{\pi,k-1}$. We know that the rotated odometer map $F_\pi$ 
maps the interval 
$H_1 = R_\pi^{-1}( [1 - 2^{-N},1))$ onto $L_1 = [0, 2^{-N})$ discontinuously.  We now note that $F_{\pi,k-1}$ maps the interval 
$H_{k} := R_{\pi,k-1}^{-1}\left(\left[\frac{2^N-1}{2^{kN}}, \frac{2^N}{2^{kN}} \right)\right)$ onto  $L_{k} = [0, \frac{1}{2^{kN}})$ discontinuously. This follows from the fact that $L_{k-1}$ is subdivided into $q2^N$ intervals of the partition $\mathcal{P}_{kN,q}$, and the multiplication by $2^{(k-1)N}$ 
maps these intervals onto $q2^N$ intervals of the partition $\mathcal{P}_{N,q}$ of $I$, in particular, $L_k$ onto $L_1 = [0, 2^{-N})$ and $H_{k}$ onto $H_1 = R_{\pi,1}^{-1}( [1 - 2^{-N},1))$.

For $j \geq 1$, denote by $C_j=[1 -2^{-(j-1)},1-2^{-j})$ the half-open intervals on which $\mathfrak a$ is continuous.

We note that for $1 \leq j \leq N$, the intervals $C_j$ are partitioned into the intervals of $\mathcal{P}_{N,q}$, and so $\mathfrak a$ is continuous on the intervals of $\mathcal{P}_{N,q}$ contained in $C_1 \cup \cdots \cup C_N = I \setminus [1-2^{-N},1)$. Then $F_{\pi,k-1}$ is continuous on any interval of $\mathcal{P}_{kN,q}$ contained in $L_{k-1} \setminus H_{k}$, and maps any such interval onto another interval of $\mathcal{P}_{kN,q}$. 
The first return map $F_{\pi,k-1}$ is discontinuous on the intervals of $\mathcal{P}_{kN,q}$ contained in $H_k$.

Denote by $t_{k,i} \geq 1$ the smallest integer such that  $F_{\pi,k-1}^{t_{k,i}-1}(I_{k,i}) \subset H_{k}$, $I_{k,i} \in \mathcal{P}_{kN,q}^{cod}$ and define
  \begin{align}\label{eq-permutationk}
      R_{\pi,k}: L_k\to L_k, \qquad I_{k,i} \mapsto R_{\pi,k-1} \circ F_{\pi,k-1}^{t_{k,i} - 1}(I_{k,i}) + 2^{kN} -1. 
  \end{align}
 Then the first return map $F_{\pi,k}^{t_{k,i}} = \mathfrak a_k \circ R_{\pi,k}$, which proves \eqref{it1-nn}.
 
 The $q$ intervals in the coding partition $\mathcal{P}_{kN,q}^{cod}$ of $L_k$ have a natural order, induced by the order 
in which they cover $L_k$. The map $R_{\pi,k}: L_k \to L_k$ is a permutation of the sets in $\mathcal{P}_{kN,q}^{cod}$, 
and so it defines a permutation $\pi_k$ of $q$ symbols. This permutation need not be the same permutation as $\pi$, however, since the group of permutations on $q$ symbols is finite, either $\pi_k \ne \pi_j$ for $j>k$, 
or there exists a smallest index $1 \leq k<j$ such that $\pi_k = \pi_j$, and then $\pi_{k} = \pi_{k+s(j-k)}$ for 
all $s \geq 1$.  This finishes the proof of the proposition.
\end{proof}

\subsection{Periodic points and the first return maps}

Since every interval $I_{k,i}$ in $\mathcal{P}_{kN,q}^{cod}$ in Proposition~\ref{lem:periodicpermutations} returns to $L_k$, 
and an orbit visits $L_k$ if and only if it visits $H_{k}$, 
we have $H_{k} =\bigcup_{i = 0}^{q-1} F_{\pi,k-1}^{t_{k,i}-1}(I_{k,i})$. 
However, the orbits of the sets in $\mathcal{P}_{kN,q}^{cod}$ may miss some of the intervals in the partition of $L_{k-1}$ 
into intervals of $\mathcal{P}_{kN,q}$. Such intervals contain points with periodic orbits. 

\begin{definition}\label{def-covering}
The first return map $F_{\pi,k}$ is {\em covering} if the orbits of the sets in $\mathcal{P}_{kN,q}^{cod}$ 
visit every element of $\mathcal{P}_{kN,q}$ in $L_{k-1}$, that is,  if
\begin{equation}\label{eq:covering} 
\bigcup_{i=0}^{q-1} \bigcup_{t=0}^{t_{k,i}-1} F^t_{\pi,k-1}(I_{k,i}) = L_{k-1}. 
\end{equation}
\end{definition}

The set in \eqref{eq:covering} is a disjoint union of intervals in $\mathcal{P}_{kN,q}$, and it follows by induction that if the 
first return maps $F_{\pi,j}$ are covering for all $1 \leq j \leq k$, then 
\begin{equation}\label{eq:covering2} 
\bigcup_{i=0}^{q-1} \bigcup_{t=0}^{h_{k,i}-1} F^t_{\pi}(I_{k,i}) = I, 
\end{equation}
where $h_{k,i}$ is the first return time of $I_{k,i}$ to $L_k$ under iterations of $F_\pi$. 

\begin{definition}\label{defn-covering}
The rotated odometer $F_\pi: I \to I$ is \emph{covering} if for every $k \geq 1$, the first return map $F_{\pi,k}$ is covering, that is, $F_{\pi,k}$ satisfies \eqref{eq:covering}.
\end{definition}

Thus if a rotated odometer $F_\pi$ is not covering, then for some $k \geq 1$, the first return map $F_{\pi,k}$ does not satisfy \eqref{eq:covering}. This means that $L_{k-1}$ contains subintervals which are not visited by the orbits of the sets in $\mathcal{P}_{kN,q}^{cod}$ under the first return map $F_{\pi,k-1}$. The points contained in these subintervals must be periodic.

As a consequence of the argument above, we deduce the following general
lemma,  without the assumption that all $F_{\pi,k}$ are covering:

\begin{lemma}\label{lem:UI}
The set of non-periodic points in $I$ satisfies
    $$I_{np} = \bigcap_{k \geq 1}  \, \bigcup_{i=0}^{q-1} \bigcup_{j=0}^{h_{k,i}-1} F^j_\pi(I_{k,i}).$$
\end{lemma}

We now use the collection $(L_k,F_{\pi,k})$ to build 
a symbolic representation of the rotated odometer.

\subsection{The construction of the Bratteli-Vershik system}\label{subsec-BV}
In this section we build an ordered 
Bratteli diagram $(V',E',<)$, determined by the
sequence $(L_k,F_{\pi,k})$ of first return maps of Section~\ref{subsec:firstreturn}, for $k \geq 1$. Note that $(L_k,F_{\pi,k})$ may have periodic points, so to prove the theorem we will have to restrict to a subdiagram of $(V',E',<)$, which is described further in Section~\ref{subsec:conjrotated}.

Recall that the interval $I$ is subdivided into $q2^{(k-1)N}$ subintervals of equal length of the partition $\mathcal{P}_{(k-1)N,q}$ of Definition \ref{defn-not}, and the subinterval $L_{k-1}$ of $I$ is comprised of the $q$ first sets of this partition. This subset of $q$ sets is denoted by $\mathcal{P}_{(k-1)N,q}^{cod}$, see \eqref{eq-codpart}, to underline its special role in the arguments. The partition $\mathcal{P}_{(k-1)N,q}^{cod}$ serves as a coding partition for the orbits in $L_{k-1}$ of the sets of the finer partition $\mathcal{P}_{kN,q}^{cod}$, subdividing $L_k \subset L_{k-1}$. These codings are used to define the substitutions $\chi_k$ for $k \geq 1$. 

The subinterval $L_{k-1}$ is partitioned into $q$ sets of $\mathcal{P}_{(k-1)N,q}^{cod}$, and into $q2^N$ smaller sets of the finer partition $\mathcal{P}_{kN,q}$. The first return map $F_{\pi,k-1}:L_{k-1} \to L_{k-1}$ maps the intervals from the partition $\mathcal{P}_{kN,q}^{cod}$ of the subset $L_k$ of $ L_{k-1}$ onto the intervals of $\mathcal{P}_{kN,q}$ in $L_{k-1}$. The latter are contained in the intervals of $\mathcal{P}_{(k-1)N,q}^{cod}$ which are ordered naturally from $0$ to $q-1$. Thus $\mathcal{P}_{(k-1)N,q}^{cod}$ codes the orbits of sets of the partition $\mathcal{P}_{kN,q}^{cod}$ of $L_k$ in $L_{k-1}$.
For $k \geq 1$ we define a substitution
 $$
 \chi_{k}(i) = e_0\dots e_{t_{k,i}-1}, \quad \text{ where } F_{\pi,k-1}^t(I_{k,i}) \subset I_{k-1,e_t} \in \mathcal{P}_{(k-1)N,q}^{cod},\ 0 \leq t < t_{k,i}.
 $$
 The substitution word $\chi_k(i)$ tracks the order of the sets of the partition $\mathcal{P}_{(k-1)N,q}^{cod}$ of $L_{k-1}$ visited by the orbit of $I_{k,i}$ under $F_{\pi,k-1}$ before returning to $L_k$. The associated matrix is given by
 $$
 M_k = (m_{i,j})_{i,j=0}^{q-1} \qquad m_{i,j} = \#\{ 0 \leq t < t_{k,i} : e_t = j\}.
 $$
 This proves item \eqref{it-main3-1} of Theorem~\ref{thm-main3l}. The properties of $(\chi_k)_{k \geq 1}$ described in Theorem~\ref{thm-main3l} follow from item \eqref{it2-nn} in Proposition~\ref{lem:periodicpermutations}, and the fact that $\chi_k$ is determined by the permutation $\pi_k$.
 
 The following property of the sequence $(\chi_k)_{k \geq 1}$ follows from the fact that every set $I_{k,i}$, contained in $L_k$, visits the interval $H_k$ before returning to $L_k$.
 
 \begin{lemma}\label{lemma-proper}
 Every substitution in the sequence $(\chi_k)_{k \geq 1}$ is proper, namely, every word $\chi_k(i)$ starts with $0$, and ends 
 with $b_k$, depending only on $k$ and not on $i \in \{0,\ldots,q-1\}$.
 \end{lemma}
 
 We now can proceed similarly to the case of a single substitution in Section~\ref{subsec:subst}.
By Lemma~\ref{lemma-proper} the sequence
$$
\rho = \lim_{k \to \infty} \chi_1 \circ \cdots \circ \chi_k(0)
$$
is well-defined (it is actually the itinerary of $0$ in $I$).
We define the  {\em $S$-adic subshift} $(X_\rho,\sigma)$ similarly to Section~\ref{subsec:subst}, formula \eqref{eq-shiftmap}, and the paragraph below \eqref{eq-shiftmap}. 

Since the sequence $(\chi_k)_{k \geq 1}$ is pre-periodic, the properties of the corresponding subshift are effectively the same as for the case of a single substitution. Namely, 
denote by $k_0$ the length of the pre-periodic part of $(\chi_k)_{k \geq 1}$, and by $p_0$ the length of one period. Then set
 \begin{align}\label{eq-bmat} B = M_{k_0+p_0} \cdot \cdots \cdot M_{k_0+1}, \end{align} 
 and
 \begin{equation}\label{eq:hhw}
 \vec w =  M_{k_0}  \cdots M_1 \cdot h^{(1)},
\end{equation}
where $\cdot$ is the matrix multiplication. The stationary Bratteli diagram with associated matrix \eqref{eq-bmat} and the vector $h^{(1)} = \vec w$ can be obtained from the diagram $(V,E,<)$ by telescoping.

We now adapt the procedure of Example~\ref{example-mainmethod} to the case of a sequence of substitutions $(\chi_k)_{k\geq 1}$.

Set $V_0' = \{v_0\}$ and $V_k' = \{ 0, \dots, q-1\}$ for $k \geq 1$, so the vertices in $V_k'$ correspond to the sets in $\mathcal{P}_{(k-1)N,q}^{cod}$ for $k \geq 1$. 
Define the set $E_{k}'$ of edges from $V_{k}'$, and the order on the edges, using the substitution $\chi_k$ as in Example~\ref{example-mainmethod}. Then the number of incoming edges to $i \in V_{k+1}'$ is equal to $|\chi_k(i)| = t_{k,i}$. 

\begin{example}\label{ex:120}
{\rm
 Let $q=3$ and $\pi = (012)$, then $R_\pi$ is the rotation over $1/3$.
The first return map
$F_{\pi,1}$ corresponds to the following substitution and associated matrix $M_1$:
$$
\chi_1: \begin{cases}
       0 \to 0221\\
       1 \to 0221\\
       2 \to 0011
      \end{cases}
      \qquad \qquad
M_1 = \begin{pmatrix}
     1 & 1 & 2 \\
     1 & 1 & 2 \\
     2 & 2 & 0
    \end{pmatrix}.
$$
Since $F_{\pi,1}$ is conjugate to $F_\pi$, we find that
$\chi_k = \chi_1$ and $M_k = M_1$ for all $k \geq 1$,
generating a substitution shift $(X_\rho, \sigma)$, where $X_\rho$ is the shift-orbit closure of the fixed
point 
$$
\rho = 0 \cdot 221 \cdot 001100110221 \cdot 0221001100110221\dots
$$
of the substitution $\chi_1$. 
The corresponding Bratteli diagram is given in Figure~\ref{fig:BV}.

\begin{figure}[ht]
\begin{center}
\begin{tikzpicture}[scale=0.8]
\filldraw(2,10) circle (4pt); \node at (1.6,10.2) {\small $v_0$};
\filldraw(0,8) circle (4pt);  \filldraw(2,8) circle (4pt); \filldraw(4,8) circle (4pt);
\draw[-, draw=black] (2,10)--(0,8); \draw[-, draw=blue] (2,10)--(2,8); \draw[-, draw=purple] (2,10)--(4,8);
\filldraw(0,6) circle (4pt);  \filldraw(2,6) circle (4pt);  \filldraw(4,6) circle (4pt); 
\draw[-, draw=black] (0,8) .. controls (-0.2,7) .. (0,6); 
\draw[-, draw=black] (4,8) .. controls (0.3,7) .. (0,6); 
\draw[-, draw=black] (4,8) .. controls (0.7,7) .. (0,6);
\draw[-, draw=black] (2,8) .. controls (1.4,7) .. (0,6);
\draw[-, draw=blue] (0,8) .. controls (0.5,7) .. (2,6); 
\draw[-, draw=blue] (4,8) .. controls (1.5,7) .. (2,6); 
\draw[-, draw=blue] (4,8) .. controls (2,7) .. (2,6);
\draw[-, draw=blue] (2,8) .. controls (2.5,7) .. (2,6);
\draw[-, draw=purple] (0,8) .. controls (0.7,7) .. (4,6); 
\draw[-, draw=purple] (0,8) .. controls (1.3,7) .. (4,6); 
\draw[-, draw=purple] (2,8) .. controls (2.7,7) .. (4,6);
\draw[-, draw=purple] (2,8) .. controls (3.3,7) .. (4,6);
\filldraw(0,4) circle (4pt);  \filldraw(2,4) circle (4pt);  \filldraw(4,4) circle (4pt);
\draw[-, draw=black] (0,6) .. controls (-0.2,5) .. (0,4); 
\draw[-, draw=black] (4,6) .. controls (0.2,5) .. (0,4); 
\draw[-, draw=black] (4,6) .. controls (0.6,5) .. (0,4);
\draw[-, draw=black] (2,6) .. controls (1.4,5) .. (0,4);
\draw[-, draw=blue] (0,6) .. controls (0.5,5) .. (2,4); 
\draw[-, draw=blue] (4,6) .. controls (1.5,5) .. (2,4); 
\draw[-, draw=blue] (4,6) .. controls (2,5) .. (2,4);
\draw[-, draw=blue] (2,6) .. controls (2.5,5) .. (2,4);
\draw[-, draw=purple] (0,6) .. controls (0.7,5) .. (4,4); 
\draw[-, draw=purple] (0,6) .. controls (1.3,5) .. (4,4); 
\draw[-, draw=purple] (2,6) .. controls (2.7,5) .. (4,4);
\draw[-, draw=purple] (2,6) .. controls (3.3,5) .. (4,4);
\node at (0,3.5) {\small $\vdots$};\node at (2,3.5) {\small $\vdots$};\node at (4,3.5) {\small $\vdots$};
\end{tikzpicture}
\caption{The Bratteli diagram of $F_{(012)}$ with incoming edges ordered left to right.}
\label{fig:BV}
\end{center}
\end{figure}
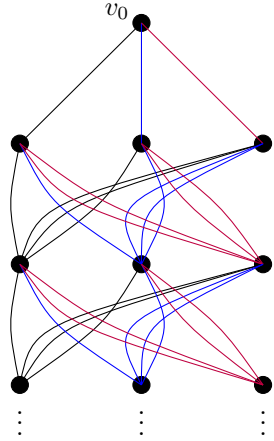 
}
\end{example}

Reinterpreting Definition~\ref{def-covering}, we can say that $\chi_k$ is \emph{covering} if $\sum_{i=0}^{q-1} |\chi_k(i)| = q2^N$, where $|\chi_k(i)|$ denotes the length of the word $\chi_k(i)$.
We also obtain an alternative expression for the number of paths from the root $v_0$ to $i \in V_{k}$, namely $h^{(k)}_i = |\chi_1 \circ \dots \circ \chi_{k-1}(i)|$.

\begin{remark}\label{rem:covering-primitive}
{\rm
 It is useful to point out that the notions of primitive and covering do not imply each other.
In Example~\ref{ex-finiteperiodic} below $\chi_1$ is primitive but not covering, whereas in Example~\ref{ex:covering-not-primitive} for every $\chi_k = \chi_1$, $k \geq 1$, the substitution 
 is covering but not primitive.
 }
\end{remark}

\begin{lemma}\label{lem:proper}
 The ordered Bratteli diagram $(V',E',<)$ constructed above has a unique
 minimal and unique maximal path, and so the Vershik map $\tau': X_{(V',E',<)} \to X_{(V',E',<)}$, constructed as in Section~\ref{subsec:BV}, is a homeomorphism.
\end{lemma}

\begin{proof}
By Lemma~\ref{lemma-proper} $e^{\min} = 0^\infty$ is the
 unique minimal path, and  
 the maximal path is $e^{\max} = b_1b_2b_3\dots$, where $b_k$ is the last letter of any word of the substitution $\chi_k$.
\end{proof}

\subsection{Conjugacy to the rotated odometer}\label{subsec:conjrotated} 
In this section we show that the aperiodic system $(I_{np}^*,F_\pi^*)$ is conjugate to a subsystem of the Bratteli-Vershik system $(X_{(V',E',<)},\tau')$,  constructed in Section~\ref{subsec-BV}. 

\begin{theorem}\label{thm:conj}
Consider the ordered Bratteli diagram $(V',E',<)$ and the Bratteli-Vershik system $(X_{(V',E',<)},\tau')$, constructed in Section~\ref{subsec-BV}. There exists a subdiagram $(V,E,<)$ of $(V',E',<)$ with associated Bratteli-Vershik system $(X_{(V,E,<)},\tau)$, such that there is a homeomorphism $\psi: I_{np}^* \to X_{(V,E,<)}$, which satisfies $\psi \circ F_\pi^*(x) = \tau \circ \psi(x)$ for all $x \in I_{np}^*$. 
\end{theorem}

\begin{proof}  Consider the sequence of first return maps $(L_k,F_{\pi,k})$, $k \geq 1$. 
Recall that for $k \geq 1$ the interval $L_k$ has a partition $\mathcal{P}_{kN,q}^{cod} = \{I_{k,i} : 0 \leq i \leq q-1\}$ 
into $q$ sets, and $|V_k'| = |\mathcal{P}_{(k-1)N,q}^{cod}| = q$.

Let $V_0 = V_0'$, and
for $k \geq 1$ let $i \in V_k$ if and only if $I_{k-1,i} \cap I_{np} \ne \emptyset$. 
Let $e \in E_k$ if and only if ${\bf s}(e) \in V_{k}$ and ${\bf t}(e) \in V_{k+1}$ for some $k \geq 0$. 
Give the edges in $E = \bigsqcup_{k \geq 0} E_k$ the order which is the restriction of the order in $E'$. Let $X_{(V,E,<)}$ be the path space 
of the subdiagram with vertex set $V = \bigsqcup_{k \geq 0} V_k$ and edge set $E$. 
Then the Vershik map $\tau'$ on $X_{(V',E',<)}$ induces the Vershik map $\tau$
on $X_{(V,E,<)}$. More precisely, for the paths $e,e' \in X_{(V,E,<)}$ 
we have $\tau(e) = e'$ if and only if there is $n \geq 1$ such that $(\tau')^n(e) = e'$ in $X_{(V',E',<)}$
and $(\tau')^m(e) \notin X_{(V,E,<)}$ for $1 \leq m < n$.
 
We now define the map $\psi: I_{np} \to X_{(V,E,<)}$ as follows:
For $x \in I_{np}$ and $k \geq 1$, there is a unique $i_k(x) \in \{0,\dots, q-1\}$ such that 
$F^j_\pi(I_{k-1,i_k(x)}) \owns x$ for some $0 \leq j < h_{k,i_k(x)}$. 
Then $i(x) := (i_k(x))_{k \geq 0}$ is the sequence of vertices in the Bratteli-Vershik diagram 
corresponding to $\psi(x)$, where $i_0(x)$ is the root of the diagram. 

We determine the sequence of edges $(e_k(x))_{k \geq 0}$ inductively. There is a single edge to each $i_1(x) \in V_1$ from $v_0$, so $e_0(x)$ is determined. Next, there is a unique $0 \leq j_1 < |\chi_1(i_2(x))|$ such that $x \in F^{j_1}_{\pi}(I_{1,i_2(x)})$,
so we take $e_1(x)$ to be the $j_1$-th incoming edge to $i_2(x) \in V_2$.
If $e_0(x),\ldots,e_{k-1}(x)$, and so $j_1, \dots, j_{k-1}$, are determined,
there is a unique $0 \leq j_k < |\chi_{k}(i_{k+1}(x))|$ such that
$F_{\pi,1}^{j_1} \circ F_{\pi,2}^{j_2} \circ \dots \circ F_{\pi,k}^{j_k}(I_{k,i_{k+1}(x)}) \owns x$, and we take $e_k$ to be the $j_k$-th incoming edge to $i_{k+1}(x) \in V_{k+1}$. This defines $\psi(x)  = (e_k)_{k \geq 0}$. 
We then extend $\psi$ from $I_{np}$ to $I_{np}^*$ by setting
$\psi(x) = \lim_{n\to\infty} \psi(x_n)$, whenever $(x_n)_{n \in \mathbb{N}}$ is a sequence in $I_{np}$
converging from the left to $x \in I^*_{np} \setminus I_{np}$.
Then $\psi$ is a continuous surjective map. Injectivity follows from the fact that the lengths of the sets in the partitions $\mathcal{P}_{kN,q}$ tend to zero with $k$, and so the orbits of any two distinct points eventually visit distinct sets of the coding partitions $\mathcal{P}_{kN,q}^{cod}$.  

By construction $\psi(\widehat x^-)$, where $\widehat x^-$ is defined in \eqref{eq-xhat}, is the maximal path of both $X_{(V,E,<)}$ 
and $X_{(V',E',<)}$, while $\psi(0)$ is the minimal path in both diagrams. We have $F_\pi^*(\widehat x^-) = 0$, 
and, since edges in $E$ are ordered according to the order of symbols in the substitutions $(\chi_k)_{k \geq 1}$, 
it follows that $\psi \circ F_\pi^* = \tau \circ \psi$.
\end{proof}

We now give examples which illustrate that the set of periodic points $I_{per}$ may be non-empty, and so restricting to a subdiagram in Theorem~\ref{thm:conj} may be necessary.

\begin{example}\label{ex-countableperiodic}
{\rm
Let $q = 7$ and $\pi = (0654321)$. Then:

$$
\begin{array}{c| c| c}
\pi \to \pi_1  & \text{Substitution}\, \chi_1 &  \text{Associated Matrix}  \\
\hline\hline
(0654321) \to (0654321) 
& \begin{cases}
   0 \to 01461360 \qquad \sum_{i=0}^6 |\chi_1(i)| = 14\\
   1 \to 0\\
   2 \to 0\\
   3 \to 0\\
   4 \to 0\\
   5 \to 0\\
   6 \to 0 
  \end{cases}  
&  \begin{pmatrix} 
   2 & 2 & 0 & 1 & 1 & 0 & 2\\
   1 & 0 & 0 & 0 & 0 & 0 & 0\\
   1 & 0 & 0 & 0 & 0 & 0 & 0\\  
   1 & 0 & 0 & 0 & 0 & 0 & 0\\  
   1 & 0 & 0 & 0 & 0 & 0 & 0\\
   1 & 0 & 0 & 0 & 0 & 0 & 0\\
   1 & 0 & 0 & 0 & 0 & 0 & 0
  \end{pmatrix}  
\\ \hline
\end{array}
$$

The substitution $\chi_1$ is not covering, because $\sum_{i=0}^6 |\chi_1(i)| = 14$ while $q2^N = 7* 2^3 = 56$. Since $\pi = \pi_1 = \pi_k$ for $k \geq 1$, every section $L_k$ contains a subinterval of periodic points, and $I_{per}$ is a countable union of intervals.

Since the symbols $2$ and $5$ do not appear in any of the words $\chi_1(i)$, $i=0,\ldots,6$, 
in the Bratteli diagram $(V',E',<)$, constructed in Section~\ref{subsec-BV}, the vertices $2$ and $5$ do not have outgoing edges. Following the construction in Theorem \ref{thm:conj} we remove the vertices $2$ and $5$, and the corresponding rows and
columns in the associated matrix. The Vershik map on the path space of the subdiagram $(V,E,<)$ we obtain is conjugate to the aperiodic subsystem $(I_{np}^*,F_\pi^*)$, and the subdiagram satisfies all the assumptions listed in the paragraph after Definition \ref{def:brat}.

}
\end{example}

\begin{example}\label{ex-finiteperiodic}
{\rm

Let $q = 5$ and $\pi = (02413)$. Then:

$$
\begin{array}{c| c| c}
\pi \to \pi_1   & \text{Substitution} \, \chi_1&  \text{Associated Matrix}  \\
\hline\hline
(02413) \to (01234) &\begin{cases}
   0 \to 044332  \qquad \sum_{i=0}^4 |\chi_1(i)| = 30\\
   1 \to 044332\\
   2 \to 044332\\
   3 \to 044332\\
   4 \to 012012 
  \end{cases} &  \begin{pmatrix} 
   1 & 0 & 1 & 2 & 2 \\
   1 & 0 & 1 & 2 & 2 \\
   1 & 0 & 1 & 2 & 2 \\
   1 & 0 & 1 & 2 & 2 \\
   2 & 2 & 2 & 0 & 0 \\
  \end{pmatrix}
  \\ \hline
\end{array}
$$

The substitution $\chi_1$ is not covering, because 
 $\sum_{i=0}^4 |\chi_1(i)| = 30$ while $q2^N = 5* 2^3 = 40$. However, for the first return system $(L_1,F_{\pi,1})$ the map $F_{\pi,1} = F_{\pi_1}$ is determined by the substitution $\pi_1 = (01234)$, 
 which is studied in detail in Example~\ref{ex:covering-not-primitive} and Proposition~\ref{prop-q51}. The associated substitution $\chi_2$ is covering,
 and we have $\chi_k = \chi_2$ for $k \geq 2$. Therefore, for $k \geq 2$ the first return systems $(L_k,F_{\pi,k})$ have no periodic points, and $I_{per}$ is a finite union of intervals. Although we do not need to reduce to a subdiagram in this example, the system still has periodic points.
 
}
\end{example}

Since by Theorem~\ref{thm-main1} the system $(I_{np},F_\pi)$ has a unique minimal subsystem, then $(I_{np}^*,F_\pi^*)$ also has a unique minimal set $(I^*_{min},F^*_\pi)$. The set $I_{min}^*$ corresponds to a simple subdiagram $(\widehat V, \widehat E)$ of $(V,E)$ with associated Bratteli-Vershik system $(X_{(\widehat V,\widehat E,<)}, \widehat \tau)$, where $\widehat \tau$ is a restriction of $\tau$ to $X_{(\widehat V, \widehat E,<)}$.

We now can prove Theorem~\ref{thm-ergmeasures} as a consequence of Theorem~\ref{thm-main3l}.

\begin{proofof}{Theorem~\ref{thm-ergmeasures}}
Since the number of vertices at each level of the subdiagram $(V,E,<)$ is bounded by $q$, by \cite[Theorem 4.3]{BKMS2013} $(I_{np}^*,F_\pi^*)$ has at most $q$ ergodic measures.
The minimal system $(I_{min}^*, F_\pi^*)$ corresponds to a primitive eventually stationary
subdiagram of $(V,E,<)$, and according to \cite[Theorem 1 and Proposition 16]{DHS1999}, it is isomorphic to
a primitive substitution shift (see the example of Proposition~\ref{prop-combinatorial})
or to  an adding machine (e.g.\ most other examples in this paper), depending on whether it is expansive or not.
It is well-known that both of these are minimal and uniquely ergodic; see e.g., \cite{Petersen} and \cite{Mi74}
(or \cite{Dur00} for a proof via linear recurrence).
\end{proofof}

\section{Entropy of rotated odometers}\label{subsec-entropy} 
In this section we prove Theorem~\ref{thm-entropy-intr}. For this we use the formula for the upper
bound for entropy of an infinite IET from \cite{DHP}, as well as the conjugacy of the aperiodic system $(I_{np}^*,F_\pi^*)$ to a Bratteli-Vershik system obtained in Theorem~\ref{thm-main3l}. 

Let  $h(I,S,\lambda)$ denote the
entropy of the dynamical system $(I,S)$ on an interval $I=[0,1)$ with respect to Lebesgue measure $\lambda$.

\begin{theorem}\label{thm-entropy-dhp}\cite{DHP}
Let $S: I \to I$ be any infinite IET, and
let $\ell_1 \geq \ell_2 \geq \ell_3 \geq \ldots $
be the lengths of the subintervals on which $S$ is continuous.
For $m \geq 1$ define $\Lambda_m = \sum_{i=1}^\infty \ell_{m+i}$. Then
 $$
  h(I,S,\lambda) \leq \liminf_{m \to \infty} \Lambda_m  \cdot \log m.
 $$
\end{theorem}

\begin{prop}\label{prop-entropy}
Let $\pi$ be a permutation on $q \geq 2$ symbols, and let $(I,F_\pi,\lambda)$ be a rotated odometer.
Then $h(I,F_\pi,\lambda) = 0$.
\end{prop}

\begin{proof} Recall that $N = \min \{n \in \mathbb{N} : 2^n \geq q\}$ and consider the partitions $\mathcal{P}_{kN,q}$ of $I$ given by Definition~\ref{defn-not}. Define $J_k = R^{-1}_\pi([1 - 2^{-kN},1))$, and consider the complement of $J_k$ in $J_{k-1}$, $k \geq 1$, with $J_0 = I$. The complement $J_{k-1} \setminus J_k$ is the union of $(2^N-1)q$ intervals of $\mathcal{P}_{kN,q}$, and we set $m_k = k(2^N-1)q$. The sum of the lengths of the intervals on which $F_\pi$ is continuous, starting from the $m_k+1$-st interval, is equal to the length of $J_k$, so $\Lambda_{m_k} = 2^{-kN}$. 
Then 
\begin{align*} 
   \Lambda_{m_k} \log m_k =  2^{-kN} \log \left( k (2^N-1)q \right)  = \frac{\log k + \log((2^N-1)q)}{2^{kN}}. 
\end{align*}
As $q$ and $N$ are fixed, $\Lambda_{m_k} \log m_k \to 0$ as $k \to \infty$,
and the statement follows. 
\end{proof}

\begin{proofof}{Theorem~\ref{thm-entropy-intr}}
By Theorem~\ref{thm-main3l} the aperiodic system $(I^*_{np},F^*_\pi)$ of the rotated odometer is conjugate to a Bratteli-Vershik system representing an eventually
periodic $S$-adic transformation. Every ergodic invariant measure of such $S$-adic transformations, just like for 
any substitution shift, has zero entropy.  Since by Proposition~\ref{prop-entropy}, Lebesgue measure also has zero entropy 
(and naturally all equidistributions on periodic orbits, if there are any, have zero entropy),
then every measure has zero entropy. By the variational principle, $(I^*_{np},F^*_\pi)$
and hence $(I^*,F_{\pi}^*)$ has zero topological entropy,
and because $(I,F_\pi)$ is a factor of $(I^*,F^*_\pi)$, it has zero topological entropy too.
\end{proofof}

\section{Ergodicity of Lebesgue measure}\label{sec:lebesgueergod}

In this section we prove Theorem~\ref{thm-nonergodicleb}, that is, we show that Lebesgue measure $\lambda$ on $I=[0,1)$ is ergodic if and only if the rotated odometer $(I,F_\pi)$ has no periodic points.

Our argument is based on the discussion of the \emph{covering} property of the rotated odometer, defined in Definition~\ref{defn-covering}. To recall, let $(I,F_\pi)$ be a rotated odometer, and $(\chi_k)_{k \geq 1}$ be the sequence of substitutions given by Theorem~\ref{thm-main3l}. The integer $N$, the sections $L_k$ and the partitions $\mathcal{P}_{kN,q}$ of $I$ into $q2^{kN}$ subintervals, 
for $k \geq 1$, used in the arguments below, were defined in Definition~\ref{defn-not}.

Recall from Definition~\ref{defn-covering} that $F_\pi$ is covering if for all $k \geq 1$ we have $\sum_{i=0}^{q-1}|\chi_k(i)| = q2^N$ and so, as in \eqref{eq:covering},
\begin{equation}\label{eq:tran}
 \bigcup_{i=0}^{q-1} \bigcup_{j=0}^{|\chi_k(i)|-1} F_{\pi,k-1}^j\Big( \Big[ \frac{i}{q2^{kN}}, \frac{i+1}{q2^{kN}} \Big) \Big)
 = \Big[0 , \frac{1}{2^{(k-1)N}} \Big) = L_{k-1}.
\end{equation}
We will see below that if the rotated odometer is covering, then Lebesgue measure is ergodic, and therefore
a.e.\ orbit (although not necessarily every) is dense in $I$.

On the other hand, if \eqref{eq:tran} fails for some $k$, then 
there is a half-open subinterval of $I$ that is not visited by the orbit of any $x \in  \left[0, \frac{1}{2^{kN}} \right)$, for $k$ sufficiently large. Since all aperiodic orbits accumulate at $0$, this shows that
no orbit is dense in $I$.
We have seen in Theorem~\ref{thm-entropy-intr} that the minimal subsystem $(I^*_{min}, F_\pi^*)$ is strictly ergodic, and therefore every orbit is dense
in $I^*_{min}$, but this does not need to hold for $I^*_{np}$.

\begin{example}\label{ex:covering-not-primitive}
{\rm Let $q = 5$ and $\pi = (01234)$. Applying the algorithm of Section~\ref{subsec-BV} we obtain:
$$
\begin{array}{c| c| c}
\pi \to \pi_1   & \text{Substitution}\, \chi_1&  \text{Associated Matrix} \\
\hline\hline
(01234)\to (01234) & \begin{cases}
   0 \to 03  \qquad \qquad \sum_{i=0}^{4} |\chi_1(i)| = q2^N = 40\\
   1 \to 03\\
   2 \to 03\\
   3 \to 03\\
   4 \to 04222111431431430420420422211143 
  \end{cases}  
& \begin{pmatrix} 
   1 & 0 & 0 & 1 & 0 \\
   1 & 0 & 0 & 1 & 0 \\
   1 & 0 & 0 & 1 & 0 \\
   1 & 0 & 0 & 1 & 0 \\
   4 & 8 & 8 & 4 & 8 \\
  \end{pmatrix}\\ \hline
\end{array}
$$

The sequence $(\chi_k)_{k \geq 1}$ is constant with $\chi_k = \chi_1$ for $k \geq 1$, see also
Proposition~\ref{prop-q51} for further properties of this example. The substitution $\chi_1$ is not primitive because $4$ does not occur in $\chi_1(i)$, $i = 0, \dots, 3$,
and the minimal system $(I^*_{min}, F_\pi^*)$ is easily seen to be the dyadic odometer.
Note also that since $\chi_1$ is covering, there is a dense orbit by Theorem~\ref{thm-nonergodicleb}, 
and also $X_{(V,E,<)} = X_{(V',E',<)}$.
However, not every orbit is forward recurrent. Indeed, consider the path $x \in X_{(V,E,<)}$ which passes through vertices 
labeled $(4,4,4,\dots)$, and where each edge is the maximal incoming edge from $4$ to $4$.
The orbit of the path $x$ is not recurrent under the forward iterations of the Vershik map $\tau$. Indeed, as soon as the orbit of $x$ 
moves away from the maximal edge $e_1$ from $4$ to $4$ at level $k$ to an edge $e_2$ joining $3$ and $4$, it never returns 
to $e_1$, and therefore to the cylinder set corresponding to the part of the path $x$ from the root $v_0$ 
to the vertex labelled by $4$ at level $k$. Hence the orbit of $x$ converges
to the minimal Cantor subset of $X_{(V,E,<)}$ corresponding to $I^*_{min}$ in forward time; its backward orbit, however, is dense in $X_{(V,E,<)}$.
}
\end{example}

A stationary Bratteli-Vershik system $(X_{(V,E,<)},\tau)$ with primitive associated matrix is minimal, see for instance \cite{Quef}. The statement below concerns the rotated odometer $(I,F_\pi)$ which may have periodic points, and for which the sequence of substitutions $(\chi_k)_{k \geq 1}$ is pre-periodic, but not necessarily stationary.

\begin{lemma}\label{lem:recurrent}
Let $\pi$ be a permutation of $q \geq 2$ symbols, and let $(I,F_\pi)$ be a rotated odometer. Let $(\chi_k)_{k \geq 1}$ be the 
associated sequence of substitutions given by Theorem~\ref{thm-main3l}.
If the matrix associated to the periodic part of $(\chi_k)_{k \geq 1}$ is primitive,
then every point in $[0,1)$ is recurrent under $F_\pi$.
\end{lemma}

\begin{proof}
Consider the partitions $\mathcal{P}_{kN,q}$, given by Definition~\ref{defn-not}, and the sets $I_{k,i} = \left[\frac{i}{q2^{kN}}, \frac{i+1}{q2^{kN}} \right)$ for $0 \leq i < q-1$, of these partitions, which subdivide the sections $L_k$, $k \geq 1$.
Applying \eqref{eq:tran}, for each $k$ and $i$ there is a positive integer $h_{k,i} \in \mathbb{N}$ such that 
$F^{h_{k,i}}_\pi(I_{k,i}) \subset \left[0, \frac{1}{2^{kN}} \right)$, and $h_{k,i}$ is the smallest positive integer with this property.
These numbers are in fact the heights of the Bratteli diagram, introduced in Definition~\ref{def:height}, see Section \ref{subsec-BV}. Define
$$
U_k := \bigcup_{i=0}^{q-1} \bigcup_{j=0}^{h_{k,i}-1} F^j_\pi(I_{k,i}),
\qquad U = \bigcap_{k \geq 1} U_k.
$$
If the substitution $\chi_k$ is not covering for some $k \geq 1$, then there are intervals
$\left[\frac{l}{q2^{kN}}, \frac{l+1}{q2^{kN}}\right)$ disjoint from $U_k$, for some $0 \leq l \leq q2^{kN}-1$. 
Points in such intervals have orbits not accumulating on $0$, 
and therefore by Proposition~\ref{prop:accum}  they are periodic, and in particular recurrent.
Hence it remains to consider points in $U$. 

By construction $U$ contains only non-periodic orbits, 
and so there is an embedding $\iota: U \to I_{np}^*$. Let $x \in U$.
Since $0$ is recurrent, we only have to consider $x > 0$.

Let $\varepsilon \in (0,x)$ be arbitrary and take $k_0$ so large that $1/(q2^{k_0N}) < \varepsilon$. 
Because $x \in U_{k_0}$, there is $m \geq 1$ and $i \in \{ 0, \dots, q-1\}$ such that $F^m_\pi$ maps $I_{k_0,i}$ 
onto an interval of length $1/(q2^{k_0N})$ containing $x$. 

Due to primititivy, we can find $k_1 > k_0$ such that for each $j \in \{ 0 , \dots, q-1\}$
the word $\chi_{k_0+1} \circ \cdots \circ \chi_{k_1}(j)$
contains $i$. Hence, there is $n_j \geq 0$ such that $F^{n_j}(I_{k_1,j}) \subset I_{k_0,i}$.
Now by Proposition~\ref{prop:accum}, $\operatorname{\text{orb}}(x)$ accumulates on $0$, so there is $m' \geq 1$
such that $F^{m'}_\pi(x) \in L_{k_1}$, and in particular $F^{m'}_\pi(x) \in I_{k_1,j}$
for some $j \in \{0,\dots, q-1\}$.
Thus $x$ returns to an $\varepsilon$-neighborhood of itself after $m'+n_j+m$ iterates.
\end{proof}

\begin{proofof}{Theorem~\ref{thm-nonergodicleb}}
We note that the system $(I^*_{np},F_\pi^*)$ has no periodic points if and only if it is covering in the sense of Definition~\ref{defn-covering}. If there are periodic points, then Lebesgue measure is not ergodic.

Assume that $(I_{np}^*,F_\pi^*)$ is covering. Let $B$ be the transition matrix associated to the periodic part of the sequence
$(\chi_k)_{k \geq 1}$, that is, $B$ is  the product of the matrices associated to the substitutions $\chi_k$ in one period of the sequence, see \eqref{eq-bmat}. Recall that $(I_{np}^*,F_\pi^*)$ has a unique minimal set. Then 
by renaming the symbols $\{ 0, \dots, q-1\}$ and telescoping we can put $B$ into the standard 
block-matrix (Frobenius) form used in \cite{BKMS2013,BKMS2010}:
\begin{equation}\label{eq:block-matrx}
B = \left( \begin{array}{c|c|c|c}
      F_1 &  0 & \dots & 0 \\
      \hline
      X_{2,1} & F_2 & \dots & 0 \\
      \hline
      \vdots & & \ddots & \vdots \\
      \hline
      X_{t,1} & X_{t,2} & \dots & F_t
    \end{array} \right),
\end{equation}  

where every non-zero submatrix $F_i$ is primitive, and for every $2 \leq i \leq t$, at least one of $X_{i,j}$ is a non-zero matrix.

For each of non-zero diagonal blocks $F_i$ (say of size $d_i \times d_i$), there is one ergodic measure $\mu_i$, namely
provided the leading eigenvalue of $F_i$ is greater than $1$, and the associated left eigenvector of $B$ can be chosen to be non-negative,
see \cite[Section 3]{BKMS2010}. There are no other ergodic measures. Note that $F_t$ is non-zero, since otherwise the symbol $q-1$ does not occur in any substitution words, which contradicts the fact that $(I_{np}^*,F_\pi^*)$ is covering. Thus $\mu_t$ is non-zero.

Denote also by $\mu_i$ the ergodic measures lifted along the inclusion $\iota:  I \to I^*$ to $(I,F_\pi)$.
By the ergodic decomposition, Lebesgue measure $\lambda = \sum_{i=1}^t a_i\, \mu_i$
for some choice of $a_i \in [0,1]$. But for all $i < t$,
the substitution associated to the first $D = \sum_{j \leq i} d_j$ symbols leaves out the remaining symbols,
and hence cannot be covering.
This means that $\mu_i$ is supported on a Cantor set of Hausdorff dimension $0 \leq \frac{\log D}{\log q} < 1$,
and hence $\mu_i$ is not absolutely continuous with respect to Lebesgue measure. This in turn means that $a_i = 0$ for $i < t$,
so $\lambda = \mu_t$ is ergodic.
\end{proofof}

\begin{example}\label{ex:measures}
{\rm
Consider again the rotated odometer with $q = 5$ and $\pi = (01234)$ from Example~\ref{ex:covering-not-primitive}. The system of substitutions $(\chi_k)_{k \geq 1}$ is constant. Interchanging the labels for the symbols $1$ and $3$, we obtain the matrix $B$ in the form \eqref{eq:block-matrx}, namely
$$
B = \left( \begin{array}{cc|cc|c}
   1 & 1 & 0 & 0 & 0 \\
   1 & 1 & 0 & 0 & 0 \\
   \hline
   1 & 1 & 0 & 0 & 0 \\
   1 & 1 & 0 & 0 & 0 \\
   \hline
   4 & 4 & 8 & 8 & 8 
   \end{array}
  \right)
\text{ with three blocks } F_1 = \begin{pmatrix} 
   1 & 1  \\
   1 & 1 \\
  \end{pmatrix}, \, F_2 = \begin{pmatrix} 
   0 & 0  \\
   0 & 0 \\
  \end{pmatrix}, \, F_3 = (8),$$
  so we expect two ergodic measures. 
Following the algorithm in \cite{BKMS2010}  and computing the eigenvalues and the left eigenvectors of $B$, 
we obtain an eigenvalue $\lambda_1 = 2$ with non-negative eigenvector $\vec{v}_1^{\ell} = (1,1,0,0,0)$,
and $\lambda_2=8$ with eigenvector $\vec{v}^{\ell}_2 = (1,1,1,1,1)$. The first eigenvalue corresponds to the ergodic 
measure $\mu_1$ supported on the minimal set $I^*_{min}$, and the second to the ergodic measure $\mu_2$ 
supported on $I_{np}^* = I^*$. By Theorem~\ref{thm-nonergodicleb}, $\mu_2$ lifts to Lebesgue measure on $I$ and, in particular, $(I,F_\pi)$ has dense orbits.
}
\end{example}

The question whether the aperiodic part $(I_{np}^*,F_\pi^*)$ of the rotated odometer always has a dense forward orbit, 
without the assumption that the Bratteli-Vershik system
is simple, remains open. The following sample substitutions
illustrate why it is hard to answer this question.
$$
\chi:\begin{cases}
      0 \to 01\\
      1 \to 01\\
      2 \to 021
     \end{cases}
     \text{ with }
B = \left( \begin{array}{cc|c}
      1 & 1 &  0  \\
      1 & 1 & 0 \\
      \hline
      1 & 1 & 1 
    \end{array} \right)
\quad \text{ or } \quad
\chi:\begin{cases}
      0 \to 01\\
      1 \to 01\\
      2 \to 0221\\
      3 \to 0331
     \end{cases}
      \text{ with }
B = \left( \begin{array}{cc|c|c}
      1 & 1 &  0 & 0 \\
      1 & 1 & 0 & 0\\
      \hline
      1 & 1 & 2 & 0 \\
      \hline
      1 & 1 & 0 & 2
    \end{array} \right).
$$
Both substitutions are not primitive, and the first gives an isolated path passing through the vertices $(2,2,2, \dots)$ 
in the Bratteli-Vershik diagram. Since for the rotated odometer systems $I^*_{np}$ is a Cantor set, this substitution cannot occur in rotated odometers.
However, for the second substitution, the path space is a Cantor set and every path 
in the Bratteli diagram is recurrent under the Vershik map, but since symbols $2$ and $3$ 
do not communicate, there is no dense orbit.

\section{Equicontinuous factors of rotated odometers}\label{sec:spec}

In this section we prove Theorems~\ref{thm-main4} and \ref{thm-main5}. More precisely, with the aim of classifying the dynamical systems $(I^*_{np},F^*_\pi)$ and more specifically $(I_{min}^*, F_\pi^*)$
up to isomorphism, we consider the spectrum of the Koopman operator $U_{F_\pi} g = g \circ F_\pi$, where $g:I \to \mathbb{R}$ is a measurable function.
The theory of eigenvalues of the Koopman operator of Bratteli-Vershik systems was developed in multiple papers, see 
for instance \cite{CDHM03,FMN1996,Fogg2002,Sol1992,BKMS2010}.

\subsection{Stationary diagrams}\label{subsec-eigenstationary} First let us assume that the sequence $(\chi_k)_{k \geq 1}$ is constant, that is, for all $k \geq 1$ we have $\chi_k = \chi_1$. Then the corresponding Bratteli-Vershik system is stationary with associated matrix $B$.
Denote $h^{(1)} =(1,1,\ldots,1)^T$, and consider the sequence of integer vectors
\begin{equation}\label{eq:h}
 h^{(n+1)} =  B^n \cdot h^{(1)}.
\end{equation}
Then the component $h^{(n)}_j$, $0 \leq j \leq q-1$, is equal to the number of paths in the Bratteli diagram from
$v_0$ to the $j$-th vertex of $V_n$, or, equivalently, to the height of the $j$-th stack at level $n$ in
the cutting-and-stacking representation of the system.

Consider the Bratteli-Vershik system $(X_{(\widehat V, \widehat E,<)}, \widehat \tau)$, which corresponds to a simple subdiagram of the diagram $(V,E,<)$, and which is conjugate to $(I_{min}^*, F^*_\pi)$. For simplicity assume that $B$ is in the Frobenius form. Then the submatrix $F_1$ of $B$ is primitive.
Since all our Bratteli diagrams are eventually stationary, we can use Host's results \cite{Host86} (see also \cite{Sol1992}) on substitution shifts to obtain a condition for the eigenvalues of the Koopman operator,
expressed in the language of the Bratteli-Vershik systems in Theorem~\ref{thm-host} below.

We say that $\zeta$ is a \emph{measurable} (resp. \emph{continuous}) \emph{eigenvalue} of the Koopman operator, if the corresponding eigenfunction is measurable (resp. continuous).

\begin{theorem}\label{thm-host}
The number $\zeta = e^{2\pi i\alpha}$ is an eigenvalue of the simple Bratteli-Vershik system  $( X_{(\widehat V,\widehat E, <)}, \widehat \tau)$ if and only if
\begin{equation}\label{eq-eigenvalue}
\lim_{n \to \infty} \zeta^{h^{(n)}_i} = 1, \, \textrm{ for }\, i=0,\ldots, q-1, 
\end{equation}
where $q = \# V_k$, $k \geq 1$, 
and the corresponding eigenfunction is continuous.
\end{theorem}

\iffalse

The following theorem can be found in \cite{CDHM03,FMN1996, Host86, Liv87, Sol1992}.

\begin{theorem}\label{thm:ev}
Let $(\widehat X_{BV}, \widehat \tau)$ be a primitive Bratteli-Vershik system, and let $\mu$ be an
invariant ergodic measure on $\widehat X_{BV}$. The number $\zeta \in \mathbb{S}^1$ is a
 \begin{enumerate}
  \item[1.] measurable eigenvalue of $U_\tau$ if and only if
  $\sum_{n \geq 1} | \zeta^{h_i^{(n)}}-1 |^2 < \infty$ for every $i$;
    \item[2.] \label{item-cts} continuous eigenvalue of $U_\tau$ if and only if
  $\sum_{n \geq 1} | \zeta^{ h_i^{(n)} } -1 | < \infty$ for every $i$.
 \end{enumerate}
\end{theorem}

Since all the eigenvalues (both rational and irrational) in our examples satisfy part 2.\ of this theorem
(in fact, it is known that eigenfunctions of primitive substitution shifts are always continuous,
see \cite{Host86}), all measure-theoretical factors of
the minimal parts $(I^*_{min}, F^*_{\pi})$ are in fact continuous factors.
If the  Bratteli diagram is stationary with adjacency matrix $B$, then Theorem~\ref{thm:ev} in fact reduces to:
\fi

For the rational eigenvalues of the Koopman operator, we have the following.

\begin{lemma}\label{lemma-divisors}\cite[Proposition 2]{FMN1996}
Let $k/d \in \mathbb{Q}$. Then $e^{2\pi i k/d}$ is a continuous eigenvalue of the simple Bratteli-Vershik system $(X_{(\widehat V,\widehat E,<)}, \widehat \tau)$ 
if and  only if $d$ divides $h^{(n)}_i$ for all $0 \leq i \leq q-1$ and all sufficiently large $n$.
\end{lemma}

Thus the rational spectrum of the Koopman operator of a Bratteli diagram consists of two parts: 
the eigenvalues $\zeta = e^{2 \pi i \alpha}$, where $\alpha$ is a common divisor of the eigenvalues of the matrix $B$, 
and so determined by the matrix $B$, and the combinatorial eigenvalues which depend on the decomposition of $h^{(1)}$ over the basis of right eigenvectors of $B$. 
We present examples of both types in this section. 

\begin{example}\label{example-dyadicod}
{\rm
Consider a Bratteli diagram where for $k \geq 0$ the vertex set $V_k$ consists of a single vertex, and the edge set $E_k$ consists of two edges, ordered from $0$ to $1$. The space path of this diagram can be identified with the space of 
infinite sequences $\Omega=\{0,1\}^{\mathbb N}$, and there are unique minimal and maximal paths, consisting only of $0$'s and $1$'s respectively. It is well-known that the von Neumann-Kakutani map $(I,\mathfrak a)$ is measurably isomorphic to the dynamical system on $\Omega$ induced by the Vershik map of the diagram, so we denote the induced map on $\Omega$ also by $\mathfrak a$. The system $(\Omega,\mathfrak a)$ is called the \emph{dyadic odometer}. 

The matrix associated to this Bratteli diagram is the matrix
\begin{align}\label{eq-bmatodometer}
  B = \begin{pmatrix} 1 & 1  \\   1 & 1 \end{pmatrix} \quad \text{ with eigenvalues } 0 \text{ and } 2.
\end{align}
Thus the rational spectrum of the Koopman operator for $(\Omega,\mathfrak a)$ consists of continuous eigenvalues $\{e^{2\pi ip/2^n}\mid p,n \geq 1\}$. 

It is also well-known that every (not necessarily minimal) dynamical systems has an equicontinuous factor, see for instance \cite[Chapter 9]{Auslander1988}. Recall for instance from \cite{Fogg2002} that if $\psi: (Y,g) \to (X,f)$ is a factor map of dynamical systems, then (continuous) eigenvalues of the Koopman operator of $(X,f)$ must be contained in the set of (continuous) eigenvalues of $(Y,g)$.  Thus the aperiodic system $(I_{np}^*,F_\pi^*)$ (resp. its minimal subsystem $(I_{min}^*,F_\pi^*)$) of the rotated odometer factors onto the dyadic odometer $(\Omega,\mathfrak a)$ if and only if the rational spectrum of the Koopman operator of $(I_{np}^*,F_\pi^*)$ (resp. of its minimal subsystem $(I_{min}^*,F_\pi^*)$) contains the set $\{e^{2\pi ip/2^n}\mid p,n \geq 1\}$, see \cite[Lemma 1.6.1]{Fogg2002}. 
} 
\end{example}

\subsection{Eventually stationary diagrams}
Recall that the sequence of substitutions $(\chi_k)_{k \geq 1}$, associated to the first return maps
$(L_k,F_{\pi,k})$ in Theorem~\ref{thm-main3l} is pre-periodic. For pre-periodic systems with non-trivial pre-periodic part the theory described in Section~\ref{subsec-eigenstationary} holds 
with $B$ and $h^{(1)}$ in \eqref{eq-bmat} and \eqref{eq:hhw}. 

The pre-periodic part represented by $M_{k_0}  \cdots M_1$ in \eqref{eq:hhw} corresponds to the first return map
of $F_{\pi,k_0}$ to $L_{k_0}$,
and the entire system $(I^*_{np},F^*_\pi)$ is Kakutani equivalent but not necessarily measurably isomorphic to the first return map $(L_{k_0},F_{\pi,k_0})$.
In general, the spectrum of a system and its first return map can be very different
(as the example from \cite[Section 4.5]{Petersen} shows rather spectacularly). The part of the rational 
spectrum determined by the common divisors of eigenvalues of the matrix $B$ is independent of $\vec{w}$, 
but the combinatorial part may depend on it. In fact, the system can have extra rational eigenvalues $e^{2\pi i k/d}$
if the entries in the matrix product $M_{k_0}  \cdots M_1$ are multiples of $d$.

\subsection{Rotated odometers with dyadic odometer factors}\label{subsec:dyadicfactor}

When $q = 2^n$, then the aperiodic 
subsystem $(I_{np}^*,F_\pi^*)$ is always conjugate to the dyadic odometer, and we study this case in detail in paper \cite{BL2020-2}. In this section we concentrate on 
the case when $q \geq 3$ and $q \ne 2^n$, for any $n \geq 1$. 

A detailed study of examples for $q = 3,5,7$ shows that  the minimal subsystem $(I_{min}^*,F_\pi^*)$ 
of the aperiodic subsystem $(I_{np}^*,F_\pi^*)$ may have the dyadic odometer as a factor. 
In some cases the dyadic eigenvalues of $(I_{min}^*,F_\pi^*)$ are determined by the matrix $B$, 
and in some cases they arise in the combinatorial part of the spectrum. 
In order to prove  Theorems~\ref{thm-main4} and \ref{thm-main5} we first describe several such examples, and then build on them to prove the theorems.

\begin{prop}\label{prop-q3}
Let $q = 3$ and let $\pi=(012)$ or $\pi = (021)$. Then $(I_{np}^*,F_\pi^*)$ 
is conjugate to the dyadic odometer, and  
$e^{2\pi i k/2^{n}}$ are continuous eigenvalues for all $k,n \in \mathbb{N}$.
\end{prop}

\begin{proof} Applying the algorithm of Section~\ref{subsec-BV} to the systems in question we obtain:

$$
\begin{array}{c| c| c| c}
\pi \to \pi_1   & \text{Substitution} \, \chi_1&  \text{Associated Matrix} & \text{char.\ polynomial}  \\
\hline\hline
(012) \to (012) & \begin{cases}
   0 \to 0221 \\
   1 \to 0221\\
   2 \to 0011
  \end{cases} & \begin{pmatrix} 
   1 & 1 & 2 \\
   1 & 1 & 2 \\
   2 & 2 & 0 \\
  \end{pmatrix}  &  \begin{array}{c} x^3 - 2x^2 - 8x \\ \text{with eigenvalues} \\ 4,-2,0   \end{array} \\ \hline
(021) \to (021)  &  \begin{cases}
   0 \to 0112211220 \\
   1 \to 0\\
   2 \to 0
  \end{cases} & \begin{pmatrix} 
   2 & 4 & 4 \\
   1 & 0 & 0 \\
   1 & 0 & 0 \\
  \end{pmatrix} & \begin{array}{c} x^3 - 2x^2 - 8x \\ \text{with eigenvalues} \\ 4,-2,0 \end{array} \\ \hline
\end{array}
$$

In both cases the associated matrices are primitive, so the system is minimal. 
Since in both cases $4$ and $-2$ are eigenvalues of the matrices, for any $p,n\geq 1$ the number 
$e^{2\pi i p/2^{n}}$ is a continuous eigenvalue of the Koopman operator. 
Consequently, the dyadic odometer is a factor of $(I_{np}^*,F_\pi^*)$. 

In both cases the Bratteli diagram has the equal incoming edge property so it is Toeplitz \cite{GJ2002},
and invertible Toeplitz shifts are odometers \cite[below Theorem 5.1]{Do2005}. Thus $(I_{np}^*,F_\pi^*)$ 
 is conjugate to the dyadic odometer.
\end{proof}

\begin{prop}\label{prop-q51}
Let $q = 5$ and let $\pi=(01234)$. Then the following is true for the aperiodic system $(I_{np}^*,F_\pi^*)$:
\begin{enumerate}
\item \label{it1-propq51} The substitutions $\chi_k = \chi_1$ for all $k \geq 1$.
\item The minimal set $I_{min}^*$ is a proper subset of $ I_{np}^*$.
\item \label{it3-propq51} For any $p,n \geq 1$, the number $e^{2\pi i p/2^{n}}$ is a continuous eigenvalue of $(I^*_{min},F^*_\pi)$, 
and $(I^*_{min},F^*_\pi)$ is conjugate to the dyadic odometer.
\item \label{it4-propq51} The dyadic odometer is the maximal equicontinuous factor of $(I^*_{np}, F_\pi^*)$.
\end{enumerate}
\end{prop}

\begin{proof}  Applying the algorithm of Section~\ref{subsec-BV} to the systems in question, we obtain:
$$
\begin{array}{c| c| c}
\pi \to \pi_1  & \text{Substitution}\, \chi_1 &  \text{Associated Matrix}  \\
\hline\hline
(01234)\to (01234) & \begin{cases}
   0 \to 03 \qquad \sum_{i=0}^{4} |\chi_1(i)| = 40  \\
   1 \to 03\\
   2 \to 03\\
   3 \to 03\\
   4 \to 04222111431431430420420422211143 
  \end{cases}   & \begin{pmatrix} 
   1 & 0 & 0 & 1 & 0 \\
   1 & 0 & 0 & 1 & 0 \\
   1 & 0 & 0 & 1 & 0 \\
   1 & 0 & 0 & 1 & 0 \\
   4 & 8 & 8 & 4 & 8 \\
  \end{pmatrix}
\\ \hline
\end{array}
$$

The system $(I_{np}^*,F_\pi^*)$ is not minimal, since $\chi_1(i) = 03$  for $i = 0,1,2,3$, and so the associated matrix of the substitution is not primitive. 
The minimal subdiagram has vertices $0$ and $3$ at each level, and the associated matrix is \eqref{eq-bmatodometer} with eigenvalues $0$ and $2$.
It follows that every $e^{2 \pi  i p/ 2^m}$, $p,m \geq 1$ is a continuous eigenvalue of the Koopman 
operator of the minimal subsystem, and by \cite[Lemma 1.6.7]{Fogg2002} the dyadic odometer $(\Omega,\mathfrak a)$ 
is a continuous factor of 
$(I_{min}^*,F_\pi^*)$, say with factor map $\psi$. 
Since the minimal subsystem has no other eigenvalues, by \cite[Theorem 1.5.6]{Fogg2002},
the restriction $\psi|_{I_{min}^*}$ to the minimal subsystem is a conjugacy to the dyadic odometer. This proves items \eqref{it1-propq51} - \eqref{it3-propq51}.

By Example~\ref{ex:measures} $(I^*_{np},F^*_\pi)$ has a measure $\mu_2$ which lifts to the ergodic measure on $I$, and the eigenvalues of $B$ are $2$ and $8$. By the algorithms in \cite{BKMS2010} for $m \geq 1$ the number $e^{2\pi i / 2^{m}}$ is a measurable eigenvalue of $(I^*_{np},F^*_\pi,\mu_2)$, and so $(I^*_{np},F^*_\pi,\mu_2)$ factors on the dyadic odometer. By \cite[Remark 6.6]{BKMS2010} the factor map is continuous.

Next we show that the dyadic odometer $(\Omega,\mathfrak a)$ is the maximal equicontinuous factor of $(I^*_{np},F_{\pi}^*)$. 
Note that $(I^*,F^*_\pi)$ has no periodic orbits, so $I^*_{np} = I^*$ and $X_{(V',E,'<)} = X_{(V,E,<)}$ in the
notation of Section~\ref{subsec:conjrotated}.
The minimal subsystem $(X_{(\widehat V,\widehat E,<)}, \widehat \tau)$ of $(X_{(V,E,<)},\tau)$ is conjugate via $\psi$ to $(\Omega,\mathfrak a)$, in particular, every point in $(\Omega,\mathfrak a)$ has a preimage in $( X_{(\widehat V, \widehat E,<)}, \widehat \tau)$ and possibly more preimages in $X_{(V,E,<)} \setminus X_{(\widehat V, \widehat E,<)}$. Recall from Example~\ref{ex:covering-not-primitive} that the forward orbit of the path $z \in X_{(V,E,<)}$, that consists of only maximal edges $\bar e_i \in E_i$ from $4 \in V_{i}$
to $4 \in V_{i+1}$, is non-recurrent under the Vershik map $\tau$, and its forward orbit $\operatorname{\text{orb}}_\tau(z)$ 
converges on the minimal subset $X_{(\widehat V,\widehat E,<)}$.

Now take any path $x = (x_i) \in X_{(V,E,<)}$ and  any $K \in \mathbb{N}$, and find another path $x_K' = (x_{K,i}')\in X_{(V,E,<)}$ such that
  $x'_{K,i} = x_i$ for $1 \leq i \leq K$ and
$x'_{K,i} = z_i = \bar e_i$ for all sufficiently large $i$. Such a path always exists since $\chi_1(4)$
contains all symbols in $\mathcal{A}$. 
Hence $\operatorname{\text{orb}}_\tau(x_K')$ accumulates on the minimal subset $X_{(\widehat V,\widehat E,<)}$.
Let $y, y_K' \in X_{(\widehat V,\widehat E,<)}$ be paths such that $\psi(y) = \psi(x)$
and $\psi(y_K') = \psi(x_K')$. By continuity of $\psi$, $y_K' \to y$ and of course also $x_K' \to x$
as $K \to \infty$.
Since $\tau^k(x')$ accumulates on the minimal set $X_{(\widehat V,\widehat E,<)}$ as $k \to \infty$, and $\psi(\tau^k(x')) = \psi(\tau^k(y'))$ for all $k$, 
uniform continuity of $\psi$ shows that $d(\tau^k(x_K'), \tau^k(y_K')) \to 0$ as $k \to \infty$.

But since $K$ is arbitrary, $x$ and $y$ are {\em regionally proximal}.
Regionally proximal points must have the same image under the factor map $\psi_{MEF}$
onto the maximal equicontinuous factor, see \cite[Proposition 2.47]{Ku}. It follows that since the dyadic odometer is an
equicontinuous factor, it has to be the maximal one. Indeed, suppose
that $(X_{(V,E,<)},\tau)$ has a larger equicontinuous factor, with factor map $\psi_{MEF}$.
If $x \in X_{(V,E,<)} \setminus X_{(\widehat V,\widehat E,<)}$ and $y \in X_{(\widehat V,\widehat E,<)}$ are points such that $\psi(x) = \psi(y)$,
then they are regionally proximal. But then $\psi_{MEF}(x)$ and $\psi_{MEF}(y)$ are also regionally
proximal. Since equicontinuous systems cannot have distinct regionally proximal points,
$\psi_{MEF} = \psi$ and $(\Omega,\mathfrak a)$ is indeed the maximal equicontinuous factor.
\end{proof}

\begin{prop}\label{prop-combinatorial}
Let $q = 7$ and $\pi = (0516234)$. Then the following is true for the aperiodic system $(I_{np}^*,F_\pi^*)$:
\begin{enumerate}
\item The substitutions $\chi_k = \chi_1$ for all $k \geq 1$.
\item The aperiodic system $(I_{np}^*,F_\pi^*)$ is minimal, i.e. $(I_{np}^*,F_\pi^*) = (I_{min}^*,F_\pi^*)$.
\item For any $p,m \geq 1$, the number $e^{2\pi i p/ 2^m}$ is a continuous eigenvalue of $(I^*_{min},F^*_\pi)$,
and so $(I^*_{min},F^*_\pi)$ has the dyadic odometer as a factor. Every rational eigenvalue $e^{2 \pi i p/2^{m}}$
belongs to the combinatorial part of the spectrum of $(I^*_{min},F^*_\pi)$.
\item The minimal subsystem $(I^*_{min}, F_\pi^*)$ has no eigenvalues $e^{2 \pi i \alpha}$ for irrational $\alpha$.
\item The system $(I_{np}^*,F_\pi^*)= (I_{min}^*,F_\pi^*)$ has a single ergodic invariant measure $\mu_1$.
\end{enumerate}
\end{prop}

\begin{proof} Applying the algorithm of Section~\ref{subsec-BV} we obtain:

$$
\begin{array}{c| c| c| c}
\pi \to \pi_1  & \text{Substitution}\, \chi_1  &  \text{Associated Matrix}& \text{char.\ polynomial}  \\
\hline\hline
\begin{array}{l}
(0516234) \\ \to (0516234)
\end{array} 
& \begin{cases}
   0 \to 0321 \quad \sum_{i=0}^{6} |\chi_1(i)| \\
   1 \to 0321 \qquad \qquad = 20\\
   2 \to 001\\
   3 \to 011\\
   4 \to 01\\
   5 \to 01\\
   6 \to 01
  \end{cases}  & \begin{pmatrix} 
   1 & 1 & 1 & 1 & 0 & 0 & 0\\
   1 & 1 & 1 & 1 & 0 & 0 & 0\\
   2 & 1 & 0 & 0 & 0 & 0 & 0\\  
   1 & 2 & 0 & 0 & 0 & 0 & 0\\  
   1 & 1 & 0 & 0 & 0 & 0 & 0\\
   1 & 1 & 0 & 0 & 0 & 0 & 0\\
   1 & 1 & 0 & 0 & 0 & 0 & 0
  \end{pmatrix} & \begin{array}{l} x^7 - 2x^6 - 6x^5 \\ \text{with eigenvalues}\\
1\pm\sqrt{7} \text{ and } \\ 0\ (\text{multiplicity } 5) \end{array}
\\ \hline
\end{array}
$$
Symbols $4$, $5$ and $6$ do not occur in the substitution words $\chi_1(i)$, $i = 0,\ldots,6$, so by Theorem \ref{thm:conj} we remove them and restrict to a subdiagram with the matrix
$$B = \begin{pmatrix} 1 & 1 & 1 & 1 \\   1 & 1 & 1 & 1 \\ 2 & 1 & 0 & 0 \\  1 & 2 & 0 & 0 \end{pmatrix} . $$
This matrix is primitive, so the the aperiodic subsystem is minimal, i.e. $(I_{np}^*,F_\pi^*) = (I_{min}^*,F_\pi^*)$. The associated matrix $B$
has eigenvalues $0$ (with multiplicity $2$) and $1 \pm \sqrt{7}$.

Although $B$ does not have eigenvalues that are multiples of $2$, the dyadic odometer is still a factor of 
the minimal subsystem by the following argument. 

\begin{lemma}\label{lemma-recurrencean}
For any $m \geq 1$, there exists $n_m \geq 1$ such that 
$2^m$ divides every component in $h^{(n)}$ for all $n \geq n_m$.
\end{lemma}

\begin{proof} Note that since the components of $h^{(1)}$ are equal, we have $h^{(n)} = (a_n,a_n,b_n,b_n)$,
for $a_n = 2(a_{n-1} + b_{n-1}), \ b_n = 3 a_{n-1}$.
In particular, $a_n$ is even for $n \geq 2$ and $b_n$ is even for $n \geq 3$, so the lemma holds for $m=1$ with $n_1 = 3$. 
Suppose there is $n_m$ such that $2^m$ divides $a_n$ and $b_n$ for $n \geq n_m$. 
Then $2^{m+1}$ divides $a_{n_m+1}$, and $b_{n_m+2}$. Also, $2^{m+1}$ divides $a_{n_m+2}$. 
The statement follows by induction, and we obtain $n_{m+1} = n_m+2$.
\end{proof}

We now look for eigenvalues $e^{2\pi i \alpha}$ with $\alpha$ irrational. Both irrational eigenvalues of $B$ are outside of the unit circle, 
and they are algebraic conjugates with minimal polynomial $x^2-7$. 
According to \cite[Corollary 1]{FMN1996}, if $e^{2 \pi i \alpha}$ is an eigenvalue, then for each algebraic 
conjugacy class of the eigenvalues of $B$, there is a polynomial $g(x) \in \mathbb{Q}[x]$ such that $g(x)$ 
takes the value $\alpha$ on each element in the conjugacy class which is outside of the unit circle, 
and such that the vector $h^{(1)}$ has non-trivial projection on the eigenspace corresponding to this element. 

In our case, since $1 \pm \sqrt{7}$ are the only non-zero eigenvalues of $B$ and $h^{(n)}$, $n \geq 1$, are vectors with integer components, $h^{(1)}$ must have non-zero projection on both eigenspaces. 
Therefore, if $e^{2 \pi i \alpha}$ is a continuous eigenvalue, then there is a polynomial $g(x)$ with 
rational coefficients such that 
$$
g(1+\sqrt{7}) = g(1-\sqrt{7}) = \alpha \notin \mathbb{Q}.
$$ 
Then also the polynomial $\widetilde g(x) = g(x+1)$ has rational coefficients $\tilde g_i$, and 
 $$\alpha = g(1\pm \sqrt{7}) = \widetilde g(\pm\sqrt{7}) = \sum_i \widetilde{g}_i(\pm \sqrt{7})^i = a \pm  b \sqrt{7}$$
for $a = \sum_{i \text{ even}} \widetilde g_i 7^{i/2} \in \mathbb{Q}$ and 
$b = \sum_{i \text{ odd}} \widetilde g_i 7^{(i-1)/2} \in \mathbb{Q}$. 
It follows that $b = 0$, and so $\alpha = a$ must be rational, which is a contradiction. 
Therefore, such polynomial $g(x)$ does not exist, and there are no eigenvalues of the 
form $e^{2\pi i \alpha}$ with $\alpha$ irrational.

Finally, since the associated matrix of the substitution is primitive, $(I_{np}^*,F_\pi^*)$ has a single ergodic measure.
\end{proof}

A similar set of arguments gives the following.

\begin{prop}\label{prop-combinatorial-1}
Let $q = 7$ and $\pi = (0361425)$. Then the following is true for the aperiodic system $(I_{np}^*,F_\pi^*)$:
\begin{enumerate}
\item The sequence $(\chi_k)_{k \geq 1}$ is constant, that is, $\chi_k = \chi_1$ for all $k \geq 1$.
\item The aperiodic system $(I_{np}^*,F_\pi^*)$ is minimal, i.e. $(I_{np}^*,F_\pi^*) = (I_{min}^*,F_\pi^*)$.
\item For any $p,m \geq 1$, the number $e^{2\pi i p/2^m}$ is a continuous eigenvalue of $(I^*_{min},F^*_\pi)$,
and so $(I^*_{min},F^*_\pi)$ has the dyadic odometer as a factor.
Every rational eigenvalue $e^{2 \pi i p/2^m}$ belongs to the combinatorial part of the spectrum of $(I^*_{min},F^*_\pi)$.
\item The minimal subsystem $(I^*_{min}, F_\pi^*)$ has no eigenvalues of the form $e^{2 \pi i \alpha}$, where $\alpha$ is irrational.
\item The system $(I_{np}^*,F_\pi^*)$ has a single ergodic invariant measure $\mu_1$.
\end{enumerate}
\end{prop}

\begin{proof}
Applying the algorithm of Section~\ref{subsec-BV} we obtain:
$$
\begin{array}{c| c| c| c}
\pi \to \pi_1  & \text{Substitution}\, \chi_1&  \text{Associated Matrix} & \text{char.\ polynomial}   \\
\hline\hline
\begin{array}{l}
(0361425) \\ \to (0361425) 
\end{array}
 &  \begin{cases}
   0 \to 0653  \qquad \sum_{i=0}^{6} |\chi_1(i)| = 40 \\
   1 \to 0653\\
   2 \to 0653\\
   3 \to 0653\\
   4 \to 013121212121212023\\
   5 \to 013\\
   6 \to 023
  \end{cases} & \begin{pmatrix} 
   1 & 0 & 0 & 1 & 0 & 1 & 1\\
   1 & 0 & 0 & 1 & 0 & 1 & 1\\
   1 & 0 & 0 & 1 & 0 & 1 & 1\\  
   1 & 0 & 0 & 1 & 0 & 1 & 1\\  
   2 & 7 & 7 & 2 & 0 & 0 & 0\\
   1 & 1 & 0 & 1 & 0 & 0 & 0\\
   1 & 0 & 1 & 1 & 0 & 0 & 0
  \end{pmatrix} 
  & \begin{array}{l} x^7 - 2x^6 - 6x^5 \\ \text{with eigenvalues}\\
1\pm\sqrt{7} \text{ and } \\
0\ (\text{multiplicity } 5) \end{array}
\\ \hline
\end{array}
$$
The symbols $4$ does not occur in the substitution words $\chi_1(i)$, $i = 0,\ldots,6$, so by Theorem \ref{thm:conj} we remove this symbol and restrict to a subdiagram with the matrix 
$$
A = \begin{pmatrix} 1 & 0 & 0 & 1 & 1 & 1 \\  1 & 0 & 0 & 1 & 1 & 1 \\  1 & 0 & 0 & 1 & 1 & 1 
\\  1 & 0 & 0 & 1 & 1 & 1 \\ 1 & 1 & 0 & 1 & 0 & 0 \\ 1 & 0 & 1 & 1 & 0 & 0 \end{pmatrix}
\quad \text{ with eigenvalues } 1\pm\sqrt{7},0 \, (\text{multiplicity } 4).
$$
This matrix is primitive, so the the aperiodic subsystem is minimal, i.e. $(I_{np}^*,F_\pi^*) = (I_{min}^*,F_\pi^*)$. 
By a similar argument as in Proposition~\ref{prop-combinatorial} the system $(I^*_{np},F^*_\pi)$ has no irrational eigenvalues.
Next note that, since the 
initial values $h^{(1)} = (1,1,1,1,1,1)^T$ are all equal, for $n \geq 1$
$$
h^{(n)} = (a_n,a_n,a_n,a_n,b_n,b_n), \quad \text{ for } a_n = 2(a_{n-1} + b_{n-1}), \ b_n = 3 a_{n-1}.
$$
The argument proceeds as in Lemma~\ref{lemma-recurrencean} to show that for every $p,m \geq 1$ the number
$e^{2\pi i p/ 2^m}$ is a (continuous) eigenvalue of the Koopman operator.
\end{proof}

\subsection{A rotated odometer without the dyadic odometer factor} \label{subsec:nodyadicfactor}
In this section we exhibit an example of a rotated odometer which does not have the dyadic odometer as a factor. 

\begin{prop}\label{prop-q5}
Let $q = 5$ and let $\pi=(02431)$. Then the following is true for the aperiodic system $(I_{np}^*,F^*_\pi)$:
\begin{enumerate}
\item \label{it1-propq5}  The sequence $(\chi_k)_{k \geq 1}$ is constant, that is, $\chi_k = \chi_1$ for all $k \geq 1$.
\item The minimal set $I_{min}^*$ is a proper subset of $ I_{np}^*$.
\item  \label{it3-propq5} 
 For all integers $m \geq 1$ and $1 \leq p < 2^m$, the number $e^{2\pi i p/ 2^m}$ is not an 
eigenvalue of $(I^*_{min},F^*_\pi)$.
So the dyadic odometer is not a factor of $(I_{min}^*,F^*_\pi)$.
\item  \label{it4-propq5} For any $a,b \in \mathbb{Q}$ and $\alpha = a + b\sqrt{5}$, there exists $s \in \mathbb{Z}$ such that the number $e^{2\pi i s \alpha}$ 
is an eigenvalue of the minimal subsystem $(I^*_{min}, F^*_\pi)$, so $(I_{min}^*,F^*_\pi)$ is not weakly mixing.
\item \label{it5-propq5} The aperiodic system $(I_{np}^*,F^*_\pi)$ with Lebesgue measure has the cyclic group with four elements as the maximal equicontinuous factor, but the factor map is not continuous.
\end{enumerate}
\end{prop}

\begin{proof}  Applying the algorithm of Section~\ref{subsec-BV} we obtain item \eqref{it1-propq5} of the Proposition:
$$
\begin{array}{c| c| c| c}
\pi \to \pi_1   & \text{Substitution} \, \chi_1&  \text{Associated Matrix} & \text{char.\ polynomial}  \\
\hline\hline
\begin{array}{l}
(02431) \\ \to (02431)
\end{array} & \begin{cases}
   0 \to 04212 \qquad \sum_{i=0}^{4} |\chi_1(i)| = 40\\
   1 \to 042\\
   2 \to 04012\\
   3 \to 040133413342013341334212   \\
   4 \to 012 
  \end{cases}  &  \begin{pmatrix} 
   1 & 1 & 2 & 0 & 1 \\
   1 & 0 & 1 & 0 & 1 \\
   2 & 1 & 1 & 0 & 1 \\
   3 & 5 & 3 & 8 & 5 \\
   1 & 1 & 1 & 0 & 0 \\
  \end{pmatrix} 
  & \begin{array}{l} x^5-10x^4+18x^3\\
+58x^2+47x+8\\
\text{with eigenvalues}\\
8,2\pm \sqrt{5},-1,-1 \end{array}
\\ \hline
\end{array}
$$
The system is covering, so $(X_{(V',E',<)},\tau') = (X_{(V,E,<)},\tau)$. 
The associated matrix has two eigenvalues of absolute value greater 
than $1$, $\lambda_1 = 2 + \sqrt{5}$ with left eigenvector 
${\displaystyle v_1^{\ell} = (1,\frac12(\sqrt{5}-1), 1, 0,\frac12(\sqrt{5} - 1))}$ and $\lambda_2 = 8$ with left eigenvector $v^{\ell}_2 = (1,1,1,1,1)$. Thus there are two invariant measures, $\mu_1$ supported on $I_{min}^*$ and $\mu_2$ supported on $I_{np}^*$.

Restricting to the minimal subset $I_{min}^*$ we obtain the symmetric matrix
$$
B =\begin{pmatrix}
 1 & 1 & 2 & 1 \\ 1 & 0 & 1 & 1 \\ 2 & 1 & 1 & 1 \\ 1 & 1 & 1 & 0
\end{pmatrix}  
\text{with eigenvalues } 2 \pm \sqrt{5}, -1, -1.
$$
Computing the decomposition of the vector $h^{(1)}$ over the basis of right eigenvectors, we obtain that $h^{(1)}$ is a linear combination of the eigenvectors 
corresponding to $2 \pm \sqrt{5}$. The eigenvalue $2 - \sqrt{5}$ is inside the unit circle. Take $g_{a,b}(x) = a (x - 2) + b = \alpha$ with $a,b \in \mathbb{Q}$, then $g(2 + \sqrt{5}) = a \sqrt{5} + b$ is irrational. Then by \cite[Corollary 1]{FMN1996} there exists $s \in \mathbb{Z}$ such that $e^{2 \pi i s (a \sqrt{5}+b)}$ is an eigenvalue of the Koopman operator of $(I_{min}^*,F_\pi^*,\mu_1)$. We showed item \eqref{it4-propq5} of the proposition.

Let $h^{(n)} = (a_n,b_n,a_n,b_n,c_n)$, and $\widehat{h}^{(n)} = (a_n,b_n,a_n,b_n)$, so $\widehat{h}^{(n)}$ is the vector of heights for the minimal subdiagram. Then 
\begin{align}\label{eq-abc-matrix}
 \begin{pmatrix} a_{n+1} \\ b_{n+1} \\ c_{n+1} \end{pmatrix}=
\left( \begin{array}{cc|c} 
        3 & 2 & 0 \\ 2 & 1 & 0 \\ \hline 6 & 10 & 8 
       \end{array} \right)
  \begin{pmatrix} a_n \\ b_n \\ c_n \end{pmatrix},
\end{align}
where the $2 \times 2$ first block corresponds to the minimal subsystem. If $e^{2 \pi i/ 2^{m}}$ is an eigenvalue of the Koopman operator of the minimal subsystem, 
then by Lemma~\ref{lemma-divisors}, $2^m$ must divide $\widehat h^{(n)}_i$ for $0 \leq i \leq 3$, and $n$ large enough. 
Inverting the first block of the matrix in \eqref{eq-abc-matrix}, we find
$$
\begin{pmatrix} a_n \\ b_n \end{pmatrix} =
\begin{pmatrix} -1 & 2 \\ 2 & -3 \end{pmatrix}
 \begin{pmatrix} a_{n+1} \\ b_{n+1}  \end{pmatrix},
$$
so $gcd(a_{n+1}, b_{n+1}) = gcd(a_n, b_n) = \dots = gcd(a_0, b_0) = 1$. 
This shows that there cannot be a combinatorial rational eigenvalue $e^{2\pi i \alpha}$
other than $\alpha = 0$, and the minimal subsystem $(I_{min}^*,F^*_\pi)$ has no rational eigenvalues. This shows item \eqref{it3-propq5}.

Now consider $(I_{np}^*,F^*_\pi,\mu_2)$ with Lebesgue measure $\mu_2$. We show that $e^{2\pi i/2^m}$ is an eigenvalue of  $(I_{np}^*,F^*_\pi,\mu_2)$ if and only if $m \in \{1,2\}$. It follows that the cyclic group with $4$ elements is a measurable factor of  $(I_{np}^*,F^*_\pi,\mu_2)$. We show also that for any other rational or irrational $\alpha$, the number $e^{2\pi i \alpha}$ is not an eigenvalue of  $(I_{np}^*,F^*_\pi,\mu_2)$. 

 By \cite[Theorem 6.3]{BKMS2010} to determine eigenvalues of $(I_{np}^*,F^*_\pi,\mu_2)$ we have to consider so-called `diamonds', which for our situation are pairs of paths in $(X_{(V,E,<)},\tau)$ of the same finite length $r \geq 1$ which start at the vertex marked by $3$ in $V_1$ and end at the vertex marked by $3$ in $V_{r+1}$. Since $3$ does not occur in any $\chi_1(i)$ for $i \ne 3$, every edge in such a path joins vertices marked by $3$ at consecutive levels, and it is sufficient to consider paths of length $1$ which are just pairs of edges $(j,j')$ between the vertices marked by $3$ in $V_1$ and $V_2$. Let $\kappa$ and $\kappa'$ be the orders of $j$ and $j'$ in the set of edges incoming to $3$ in $V_2$. Then by \cite[Lemma 6.4]{BKMS2010} if $2^{m}$ divides $(\kappa - \kappa')h_3^{(n)}$ for all sufficiently large $n \geq 1$ and for all diamonds $(j,j')$, then $e^{2 \pi  i 2^{-m}}$ is an eigenvalue for $(I_{np}^*,F^*_\pi,\mu_2)$, where $h_3^{(n)}$ is the height of the $3$-rd stack in $V_n$. Since the word $\chi_1(3)$ contains a subword of two consecutive $3$'s, we conclude that $e^{2 \pi  i 2^{-m}}$ is an eigenvalue if and only if $2^{m}$ divides $h_3^{(n)} = c_n$ for all sufficiently large $n$.

Since $a_n$ and $b_n$ are always odd, we can write $a_n = 2u_n +1$ and $b_n = 2v_n+1$. Then
  $$c_{n+1} = 6a_n + 10 b_n + 8c_n = 2(6u_n+3 + 10v_n + 5 + 4c_n) = 4(3 u_n + 5v_n + 2c_n + 4),$$
which shows that $e^{2\pi i /2}$ and $e^{2\pi i /4}$ are measurable eigenvalues of the Koopman operator for $(I_{np}^*,F^*_\pi,\mu_2)$; they are not continuous  because $(I^*_{min},F^*_{\pi})$ doesn't have these eigenvalues. We note that $c_{n+1}$ is divisible by $8$ if and only if the expression in the parenthesis in the formula for $c_{n+1}$ is even. This can happen only if $u_n$ and $v_n$ are both odd, or they are both even.

We have that 
 $$a_{n+1} = 3a_n + 2b_n = 3(2u_n +1) + 2(2v_n+1) = 6u_n +3 + 4v_n +2,$$
so $u_{n+1} = 3u_n +2v_n + 2$, which shows that $u_{n+1}$ is even if and only if $u_n$ is even. Since $a_1 = 1$ and so $u_1 = 0$ is even, we conclude that $u_n$ is always even. Similarly,
 $$b_{n+1} = 2a_n + b_n = 2(2u_n +1) + 2v_n+1 = 4u_n + 2 + 2v_n +1,$$
so $v_{n+1} = 2u_n + 1 + v_n$, which shows that $v_{n+1}$ is even if $v_n$ is odd, and $v_{n+1}$ is odd if $v_n$ is even.  Since $v_1 = 0$ is even, it follows that for $k \geq 1$ the height $c_{2k}$ is not divisible by $8$, and so $e^{2\pi i /2^m}$ for $m \geq 3$ is not an eigenvalue of $(I_{np}^*,F^*_\pi,\mu_2)$.

To show that there are no other rational eigenvalues, let $p \geq 3$ be a prime. 
Suppose by contradiction that $e^{2 \pi i/p} \to 1$, so $p|c_n$ for $n$ sufficiently large.

Now note that for $n \geq 1$ we have $a_n = F_{3n}$ and $b_n = F_{3n-1}$, where $F_n$ is the $n$-th Fibonacci number. The sequence of Fibonacci numbers $(F_n \bmod p)_n$ is periodic,
therefore $(6a_n+10b_n) \bmod p$ is also periodic, and there are infinitely many $n$'s such that
$$
(6a_n+10b_n) \bmod p \equiv (6a_0+10b_0) \mod p = 16 \bmod p \not\equiv 0 \bmod p.
$$
Recalling that $c_{n+1}=6a_n+10b_n+8c_n$, we find that $c_n$ and $c_{n+1}$ cannot be simultaneously divisible by $p$, so $e^{2 \pi i /p}$ cannot be an eigenvalue.

To show that $\alpha = u+v \sqrt{5}$, $u \in \mathbb{Z}$, $v \in \mathbb{N}$,
is not an eigenvalue of $(I^*_{np},F^*_\pi,\mu_2)$, first note that by subtracting $u-2v$, it suffices to verify
$\alpha = v(2+\sqrt{5})$.
The first block in formula \eqref{eq-abc-matrix} is a Pisot matrix. In fact, it is the third power of
$\binom{1 \ 1}{1 \ 0}$, 
with eigenvalues $\delta_+ = 2+\sqrt{5}$ and $\delta_- = 2-\sqrt{5} \in (-\frac12,0)$. 

By standard computation, we find
\begin{eqnarray}\label{eq-an}
 a_n &=& \frac{5+3\sqrt{5}}{10} \delta_+^n + \frac{5-3\sqrt{5}}{10} \delta_-^n \nonumber \\
 b_n &=& \frac{5+\sqrt{5}}{10}\delta_+^n +  \frac{5-\sqrt{5}}{10} \delta_-^n\\ 
 c_n &=& 5 \cdot 8^n - \frac{10+4\sqrt{5}}{5}\delta_+^n - \frac{10-4\sqrt{5}}{5} \delta_-^n.\nonumber
\end{eqnarray}
Therefore
\begin{eqnarray*}
 \alpha c_n &=& v \delta_+\left(5 \cdot 8^n - \frac{10+4\sqrt{5}}{5}\delta_+^n - \frac{10-4\sqrt{5}}{5} \delta_-^n\right) \\
 &=& v\left( 5 \cdot 8^n (\delta_+-8) + 5 \cdot 8^{n+1} - \frac{10+4\sqrt{5}}{5}\delta_+^{n+1} 
 - \frac{10-4\sqrt{5}}{5}\delta_-^{n+1} - \frac{10-4\sqrt{5}}{5}(\delta_+-\delta_-) \delta_-^n\right) \\
 &=& v\left( 5 \cdot 8^n (\sqrt{5}-6) + c_{n+1} + 4(2-\sqrt{5}) \delta_-^n\right).
\end{eqnarray*}
Thus the distance to the nearest integer is
$$\| \alpha c_n \| = \| 5 \cdot 8^n v\sqrt{5} + 4v(2-\sqrt{5}) \delta_-^n\|
\geq \| 5v \cdot 8^n \sqrt{5}\| - \|4v \delta_-^{n+1}\|.$$
Let $\varepsilon_n \in (-\frac12,\frac12)$ be the fractional part of $5v \cdot 8^n \sqrt{5}$.
Then for $|\varepsilon_n| \leq \frac{1}{16}$ we have $\varepsilon_{n+1} = 8 \varepsilon_n$, so $|\varepsilon_n|$ increases in $n$ until
$|\varepsilon_n| > \frac{1}{16}$. Therefore 
$\| 5v \cdot 8^n \sqrt{5}\| \not\to 0$ and hence neither does $\| \alpha c_n\| \to 0$,
so $e^{2\pi i \alpha}$ is not an eigenvalue of the global system.

Finally, we show that there is no other eigenvalue for $\mu_2$.
Say $e^{2\pi i \alpha}$ for $\alpha \notin \mathbb{Q}[\sqrt{5}]$ is an eigenvalue.
Then,
$$
\| \alpha(6a_n+10b_n) \|^2 = \| \alpha c_{n+1} - 8\alpha c_n\|^2 \leq 
81 \max\{ \| \alpha c_{n+1} \|^2 , \| \alpha c_n\|^2\}
$$
is summable in $n$.
Using $a_n = F_{3n}$ and $b_n = F_{3n-1}$, where $F_n$ are the Fibonacci numbers, 
\begin{eqnarray*}
\alpha(6a_n+10b_n)&=& 2\alpha a_n (3+5\frac{b_n}{a_n} )) \\
&=& 2\alpha a_n (3 + 5(\sqrt{5}-1) ) + 5\alpha \sqrt{5} (2 - \sqrt{5})^n\\
&=& \beta a_n + o(2^{-n}) \qquad\qquad \textrm{ for } \beta = 2\alpha(5\sqrt{5}-2).
\end{eqnarray*}
Therefore
$\|\beta (a_{n+1} +  a_n) \|$ is summable. 
By \eqref{eq-an} we have 
  $$\beta (a_{n+1} +  a_n) = \beta(3+\sqrt{5}) a_n + o(2^{-n}) =: \gamma a_n + o(2^{-n}),$$
and we know (by \eqref{eq-eigenvalue} and since $a_n = F_{3n}$ for the Fibonacci numbers $F_n$) 
that $\| \gamma a_n\|$ is only square summable if $\gamma \in \mathbb{Q}[\sqrt{5}]$.
Hence $(I^*_{np}, F^*_{\pi})$ has no irrational eigenvalues. This shows item \eqref{it5-propq5}.
\end{proof}

\subsection{Proofs of Theorems~\ref{thm-main4} and \ref{thm-main5}} \label{subsec-proof45}

\begin{proofof}{Theorem~\ref{thm-main4}}
The minimal subsystems in Propositions~\ref{prop-q51} and \ref{prop-q5} provide examples for items \eqref{main4-1} and \eqref{main4-2} of Theorem~\ref{thm-main4}. From one example, it is always possible to construct infinitely many examples
by doubling $q$ and changing the permutation $\pi$ to
$$
\pi':\{ 0, \dots, 2q-1\} \to \{ 0, \dots, 2q-1\}, \quad 
\pi'(i) = \begin{cases}
            \pi(i-q) & \text{ if } i \geq q\\
            i+q &\text{ if } i <  q
          \end{cases}
$$
because then the first return map of $F_{\pi'}$ to ${[0,1/2)}$ is conjugate to $F_\pi$ on $[0,1)$ via the scaling $h(x) = 2x$. This proves the theorem.
\end{proofof}

\begin{proofof}{Theorem~\ref{thm-main5}} Item \eqref{main5-1} is proved in item \eqref{it4-propq51} of Proposition~\ref{prop-q51}, 
and item  \eqref{main5-2} of the theorem is proved in \eqref{it5-propq5} of Proposition~\ref{prop-q5}.
\end{proofof}

\end{document}